\documentclass[a4paper,10pt,twoside]{amsart}
\usepackage[left=2.7cm,right=2.7cm,top=3.3cm,bottom=3cm]{geometry}
\usepackage[english]{babel}
\usepackage[latin1]{inputenc}
\usepackage{amsmath}
\usepackage{amsfonts}
\usepackage{amsthm}
\usepackage{amscd}
\usepackage{amssymb}
\usepackage[all]{xy}
\usepackage{graphicx}
\usepackage[usenames,dvipsnames]{color}

\theoremstyle{plain}
\newtheorem{thm}{Theorem}[section]
\newtheorem{cor}[thm]{Corollary}
\newtheorem{lem}[thm]{Lemma}
\newtheorem{prop}[thm]{Proposition}

\theoremstyle{definition}
\newtheorem{dfn}[thm]{Definition}
\newtheorem{eg}[thm]{Example}
\newtheorem{rmk}[thm]{Remark}
\newtheorem{qsts}[thm]{Questions}
\newtheorem{rmks}[thm]{Remarks}

\newenvironment{prf}{\begin{proof}[Proof]}{\end{proof}}
\newenvironment{skprf}{\begin{proof}[Sketch of proof]}{\end{proof}}

\renewcommand{\ge}{\geqslant}
\renewcommand{\le}{\leslant}

\newcommand{\field}[1]{\mathbb{#1}}

\newcommand{\F}{\field{F}}

\newcommand{\M}{\field{M}}
\newcommand{\N}{\field{N}}

\newcommand{\Q}{\field{Q}}
\newcommand{\R}{\field{R}}
\newcommand{\s}{\field{S}}
\newcommand{\T}{\field{T}}

\newcommand{\Z}{\field{Z}}

\newcommand{\za}{\hat{\Z}}

\newcommand{\xa}{\hat X}

\newcommand{\hpi}{\hat\pi}
\newcommand{\hvp}{\hat v_p}

\DeclareMathOperator{\id}{id}

\DeclareMathOperator{\supp}{supp}

\newcommand{\cala}{\mathcal A}
\newcommand{\calb}{\mathcal B}
\newcommand{\calc}{\mathcal C}
\newcommand{\cald}{\mathcal D}

\newcommand{\calg}{\mathcal G}

\newcommand{\cali}{\mathcal I}

\newcommand{\calk}{\mathcal K}

\newcommand{\calp}{\mathcal P}

\newcommand{\cals}{\mathcal S}
\newcommand{\calt}{\mathcal T}

\newcommand{\gotn}{\mathfrak n}

\renewcommand{\ge}{\geqslant}
\renewcommand{\le}{\leqslant}

\newcommand{\hooklongrightarrow}{\lhook\joinrel\longrightarrow}

\newcommand{\interior}[1]{%
	{\kern0pt#1}^{\mathrm{o}}%
}

\begin{document}

\title[Coset topologies on $\Z$ and arithmetic]{Coset topologies on $\Z$ and arithmetic applications}

\author[Longhi]{Ignazio Longhi}
\email{ignazio.longhi@unito.it}

\author[Mu]{Yunzhu Mu}
\email{nancy.mu20@imperial.ac.uk}

\author[Saettone]{Francesco Maria Saettone}
\address{Department of Mathematics, Ben-Gurion University of the Negev, Be'er Sheva, Israel}
\email{saettone@post.bgu.ac.il}


\begin{abstract}  We provide a construction which covers as special cases many of the topologies on integers one can find in the literature. Moreover, our analysis of the Golomb and Kirch topologies 
inserts them in a family of connected, Hausdorff topologies on $\Z$, obtained from closed sets of the profinite completion $\za$. We also discuss various applications to number theory. \end{abstract}

\maketitle

{\flushright{\em{To David Cariolaro, in memoriam}}}

\section{Introduction}

In 1955, Furstenberg \cite{fur} produced a new proof that there are infinitely many prime numbers, by a reasoning based on endowing the ring of integers $\Z$ with a topology. Not much later, Golomb
\cite{gol1} showed that the argument also worked with a coarser topology, which moreover made the positive integers into a countable connected Hausdorff space. A further coarsening, introduced by 
Kirch \cite{krch}, had the additional property of being locally connected. These topologies, especially Golomb's one (and its generalizations to other rings), have received a significant amount of 
attention in recent years: see for example \cite{aadm}, \cite{bmt}, \cite{bspt}, \cite{bst}, \cite{cllp}, \cite{orum}, \cite{sp}, \cite{st}, \cite{szc16} and \cite{szsz} (just to mention only works 
which appeared after 2015). Topologies on the integers were also introduced (and arithmetic applications discussed) in sundry other works, such as \cite{brou}, \cite{pol}, \cite{rzz}, \cite{szc13} 
and more. In particular, Broughan provided in \cite{brou} a construction covering most of the ones cited above.

In the current paper, we examine these topologies by an approach which seems to have been somehow neglected in the literature (with the partial exception of \cite{orum}): namely, we focus our 
attention on the system of quotient rings $\Z/n\Z$ and on the natural embedding of $\Z$ into their product. This gives indeed valuable new insight, as we hope is nicely illustrated by our discussion 
of connectedness properties in subsection \ref{ss:BG} (partially inspired by the introduction of Brown spaces in \cite{cllp} and by a curious duality mentioned in \cite{szc13}). More precisely, in 
\S\ref{sss:brwntyp} we insert the Golomb and Kirch examples in a vast family of topologies on the integers and prove that all of them are connected (Theorem \ref{t:brwntyp}). In \S\ref{sss:BGS}, we 
define a big subfamily for which we give simple necessary and sufficient conditions for being locally connected (Theorem \ref{t:glmkrc}) and Hausdorff (Proposition \ref{p:glmHd}). Note also that we 
only deal with $\Z$, to reach a larger audience, but our methods require only trivial changes in language in order to work for any residually finite Dedekind domain and are likely to be useful also 
for more general rings.

Embedding the integers in the product of the $\Z/n\Z$'s leads naturally to
$$\za:= \varprojlim\Z/n\Z\,,$$
the profinite completion of $\Z$ (for readers unfamiliar with it, the construction of $\za$ and its main properties will be recalled in subsections \ref{ss:za} and \ref{ss:zatplg}). Actually, $\za$ 
is somehow implicit in the Furstenberg topology, but, as we are going to show, considering this ring casts more light also on the other topologies. In particular, \S\ref{sss:brsystr} will illustrate 
that the connectedness properties mentioned above have their ultimate root in a very simple trick, namely the existence of points of $\za$ which, in the new topologies, are contained in every 
nonempty closed subset (in the cases considered by Golomb and Kirch, these points are the topologically nilpotent elements of $\za$). Furthermore, as discussed in Remark \ref{r:tree}, 
our ``profinite'' analysis seems to imply that the deeper nature of these topologies is much more combinatorial than algebraic or arithmetic (maybe in contrast with what could appear from a naive 
reading of works like \cite{sp} and \cite{bst}). The goal of classifying all connected Hausdorff topologies on a countable set looks still quite out of reach, but we hope that \S\ref{sss:brsystr} will 
play a role in making it more accessible. \\

What about number theory? As Golomb discovered, the eponymous topology allows a reformulation of Dirichlet's theorem on primes in arithmetic progressions as a density statement; perhaps motivated by 
this, he later suggested some other applications to the study of primes (\cite[Section III]{gol2}). However, he also expressed skepticism about the potential of such topological techniques 
``without the introduction of powerful new ideas and methods'' (\cite[page 181]{gol2}; see also the concluding sentence of that paper). 

We claim that looking at $\za$ can be one of those ideas. Actually, the original goal of this paper was to study the closure in $\za$ of arithmetically interesting subsets of $\Z$, but our perusal of 
literature showed that the profinite viewpoint had much to say also on the general topic of topologies on the integers: so here we will content ourselves with showing how Dirichlet's theorem has 
a natural interpretation as a statement on the closure in $\za$ of the set of primes (Theorem \ref{t:23}) and briefly explain the connection with Euler's proof that this set is infinite (Remark 
\ref{r:elrfrst}) and with \cite[Section III]{gol2} (\S\ref{sss:gol2S3}). Moreover, in section \ref{s:sprnt} we will discuss the relation between $\za$ and the supernatural numbers, whose topological 
properties were nicely used by Pollack in \cite{pol} to recover various results about perfect numbers and the like. (The potential of looking at $\za$ for the purposes of elementary number theory was 
stressed by Lenstra in \cite{lenstr1} and \cite{lenstr}.) A second paper, more focused on closures in $\za$, is in preparation and will appear in the (hopefully not far) future. See also \cite{DL} for 
a first study of what can be achieved by more advanced methods.

\subsection*{Acknowledgments} This work has its origins in an undergraduate research project in summer 2016, in the S.U.R.F. framework of Xi'an Jiaotong - Liverpool University, which we thank for 
financial funding and logistic support. The third author also thanks the Department of Mathematics ``Giuseppe Peano'' of Universit\`a di Torino for a grant to visit XJTLU in 2016. 
The first and the third author have written the paper expanding the ideas developed in the 2016 research, to which the second author gave a significant contribution. We give our warmest thanks also 
to Du Jiayi, who was an active participant in the original project. Finally, we thank Alejandro Vidal-Lopez and Simon Lloyd for the moral support they gave to our  group in 2016, and Marcin 
Szyszkowski for providing us with a copy of \cite{szsz}.

Last but not least, our thanks to the referee, whose work made the paper more precise and readable.

\section{Some exotic topologies on $\Z$ and $\N$}

\subsection{Some notation} Throughout the paper, we are going to use the following notations: \begin{itemize}
\item $\N=\{0,1,\dots\}$ is the set of natural numbers;
\item $\N_1:=\N-\{0\}$ is the multiplicative monoid of positive integers;
\item $[a]_n$ is the reduction modulo $n\in\N_1$ of $a\in\Z$;
\item $\calp\subset\N$ is the set of prime numbers;
\item for any $n\in\Z$, its support is
$$\supp(n):=\{p\in\calp:p\text{ divides }n\}\,;$$
\item for $p\in\calp$, the $p$-adic valuation $v_p\colon\Z\rightarrow\tilde\N:=\N\cup\{\infty\}$ is defined by 
$$v_p(n):=\max\{k:p^k\text{ divides }n\}\;\;\;\text{ if }n\neq0$$
and $v_p(0):=\infty$. \end{itemize}

We will be dealing with many different topologies on a same set. When needed, we will write $\overline A^\calt$ (instead of $\overline A$) to specify that the closure of $A$ is taken 
with respect to the topology $\calt$. Also, we shall usually denote by the same symbol a topology and its restriction to a subset: given $Y\subset X$, the subspace of $(X,\calt)$ is $(Y,\calt)$. 
In the topologies we are going to consider, many open sets will be singletons: in order to lighten notation, we shall say that a point $x$ is open to mean that the singleton $\{x\}$ is an open 
subset.

Rings in this paper will always be unitary and commutative (it most cases, $\Z$ and its quotients). The group of units of a ring $R$ will be denoted by $R^*$.\\

For any $n\in\N_1$, let $\pi_n\colon\Z\twoheadrightarrow\Z/n\Z$, $a\mapsto[a]_n$, be the quotient map. If $n$ divides $m$, then we also have a ring homomorphism 
$\pi_n^m\colon\Z/m\Z\twoheadrightarrow\Z/n\Z$ so that $\pi_n=\pi_n^m\circ\pi_m$. Combining together all the maps $\pi_n$\,, one obtains an injection
\begin{equation} \label{e:dfniota} \iota\colon\Z\hooklongrightarrow\prod_{n\in\N_1}\Z/n\Z \end{equation}
by $x\mapsto(\pi_n(x))_n$\,.

\subsection{P\'eij\'i topologies} For each $n\in\N_1$ fix a topology $\calt_n$ on the finite set $\Z/n\Z$ and denote by $\prod\calt_n$ the product topology on $\prod_{n\in\N_1}\Z/n\Z$. Then $\iota$ 
induces a topology $\calt:=\iota^*\prod\calt_n$ on $\Z$, characterized as the coarsest topology such that all the maps $\pi_n$ are continuous. We recall that a base for $\calt$ is given by sets of 
the form
\begin{equation} \label{e:pjbs} U=\bigcap_{j=1}^k\pi_{n_j}^{-1}(U_j) \end{equation}
where $k\in\N$, $n_1,\dots,n_k\in\N_1$ and $U_j\in\calt_{n_j}$\,.

We call {\em p\'eij\'i topologies}\footnote{From the Chinese word for ``coset''. The name ``coset topologies'' already appears in the literature with a more restrictive meaning (see e.g. 
\cite[\S2.1]{mrkprb}). In particular, a coset topology has a subbase consisting of cosets, while our definition allows for non-trivial topologies such that no coset is open, as in example 
\eqref{e:egt6} below.} those topologies on $\Z$  which can be obtained by a family $\{\calt_n\}$ as above.

\begin{rmk} We will work only on $\Z$, but the construction above is easily extended to any unitary commutative ring $R$ which is not a field. Namely, let $\cali(R)$ be the set of non-zero ideals of 
$R$ (in the case of the ring $\Z$, there is an obvious identification of $\cali(\Z)$ with $\N_1$). Assume that $\cali(R)$ is nonempty and a topology is given on $R/I$ for each $I\in\cali(R)$. Then 
$\prod_IR/I$ has the product topology, which can be pulled back to a p\'eij\'i topology on $R$ via the natural map $R\rightarrow\prod_IR/I$. \end{rmk}

\subsubsection{Equivalence classes and maximal families} Different families $\{\calt_n\}$ and $\{\calt_n'\}$ can originate the same topology on $\Z$: if this happens, we say that 
$\{\calt_n\}$ and $\{\calt_n'\}$ are {\em equivalent}.

There is a natural order: a family $\{\calt_n\}$ is finer than $\{\calt_n'\}$ if $\calt_n$ is finer than $\calt_n'$ for each $n\in\N_1$\,. We say that $\{\calt_n\}$ is {\em maximal} if it is the 
finest representative in its equivalence class. 

\begin{prop} \label{p:fnsttplg} Every equivalence class of families $\{\calt_n\}$ as above contains a maximal representative, which is obtained by giving to each $\Z/n\Z$ the quotient topology with 
respect to $\calt$. \end{prop}

\begin{prf} Let $\calt$ be induced by a family $\{\calt_n\}$ as above. The quotient topology $\calt_n'$ on $\Z/n\Z$ is defined by postulating that $A\subseteq\Z/n\Z$ is open if and 
only if $\pi_n^{-1}(A)$ is open in $\calt$. It is obvious that $\{\calt_n'\}$ is equivalent to $\{\calt_n\}$  and that it is the finest family in its class. \end{prf}

\begin{cor} \label{c:fnsttplg} Assume $\{\calt_n\}$ is maximal. Then the maps $\pi_n^{nk}$ are all continuous. \end{cor}

\begin{prf} Assume $A\subseteq\Z/n\Z$ is open with respect to $\calt_n$\,. Then the equality
$$\pi_n^{-1}(A)=\pi_{nk}^{-1}\big((\pi_n^{nk})^{-1}(A)\big)$$
shows that $(\pi_n^{nk})^{-1}(A)$ is open in the topology defined on $\Z/nk\Z$ in the proof of Proposition \ref{p:fnsttplg}. \end{prf}

\begin{lem} \label{l:mxmcrt} The family $\{\calt_n\}$ is maximal if the maps $\pi_n^{nk}$ are all continuous and open. \end{lem}

\begin{prf} Fix $n_0\in\N_1$ and take $A\subseteq\Z/n_0\Z$ such that $\pi_{n_0}^{-1}(A)$ is open in $\calt$. By Proposition \ref{p:fnsttplg} we have to show that $\calt_{n_0}$ contains $A$. The 
definition of $\calt$ yields that $\pi_{n_0}^{-1}(A)$ is a union of sets $U$ as in \eqref{e:pjbs}. For such a $U$, let $n$ be a common multiple of $n_0,n_1,\dots,n_k$ and consider the sets
$$V_j:=(\pi_{n_j}^n)^{-1}(U_j)\subseteq\Z/n\Z\,,$$ 
with $j\in\{1,...,k\}$. Continuity of the maps $\pi_{n_j}^n$ implies that each $V_j$ is open and thus so is
$$V:=\bigcap_{j=1}^kV_j\,.$$
Moreover $\pi_{n_j}=\pi_{n_j}^n\circ\pi_n$ yields
\begin{equation} \label{e:aprtn} U=\bigcap_{j=1}^k\pi_{n_j}^{-1}(U_j)=\bigcap_{j=1}^k\pi_n^{-1}\big((\pi_{n_j}^n)^{-1}(U_j)\big)=\bigcap_{j=1}^k\pi_n^{-1}(V_j)=\pi_n^{-1}(V) \end{equation}
and, by the hypothesis that each $\pi_m^{mk}$ is open, it follows that
$$\pi_d(U)=\pi_d^n(V)$$
is open in $\calt_d$ for every $d$ dividing $n$. In particular, this shows that $\pi_{n_0}(U)$ is in $\calt_{n_0}$ and one concludes observing that $A$ is a union of such images. \end{prf}

\begin{rmks} \label{r:mxmfml} \begin{itemize} \item[]\end{itemize}
\noindent{\bf 1.} The criterion of Lemma \ref{l:mxmcrt} is sufficient to ensure maximality of a family, but is not necessary. E.g., take 
\begin{equation} \label{e:egt6} \calt_n:=\begin{cases}\big\{\emptyset,\{[0]_6,[1]_6\},\Z/6\Z\big\} & \text{ if }n=6\,;\\ \big\{\emptyset,(\pi_6^n)^{-1}(\{[0]_6,[1]_6\}),\Z/6\Z\big\}& \text{ if 
}n\in6\N_1\,;\\ \{\emptyset,\Z/n\Z\} &\text{ otherwise}\,. \end{cases} \end{equation}
One sees immediately that \eqref{e:egt6} defines a maximal family such that the map $\pi_3^6$ is not open. \\
\noindent{\bf 2.} Equality \eqref{e:aprtn} shows also that if all maps $\pi_n^{nk}$ are continuous then the sets $\pi_n^{-1}(U)$, with $n$ varying in $\N_1$ and $U$ open in $\calt_n$, form a base 
for the p\'eij\'i topology resulting from $\{\calt_n\}$. In particular this holds when $\{\calt_n\}$ is maximal, by Corollary \ref{c:fnsttplg}. \end{rmks}

\begin{lem} Assume $\{\calt_n\}$ is maximal. Then the closure of $S\subseteq\Z$ is $\bigcap_n\pi_n^{-1}\big(\overline{\pi_n(S)}\big)$. \end{lem}

\begin{prf} Since $S$ is a subset of $\cap_n\pi_n^{-1}\big(\overline{\pi_n(S)}\big)$ and the latter is closed, we get the inclusion
$$\overline S\subseteq\bigcap_{n\in\N_1}\pi_n^{-1}\big(\overline{\pi_n(S)}\big)\,.$$
Now take $z\in\cap_n\pi_n^{-1}\big(\overline{\pi_n(S)}\big)$ and let $U$ be a neighbourhood of $z$. By Remark \ref{r:mxmfml}.2, we can assume $U=\pi_n^{-1}(V)$, with $V$ open in $\calt_n$\,. Then 
$V$ is a neighbourhood of $\pi_n(z)$ and, since the latter belongs to $\overline{\pi_n(S)}$, it follows that there is some $y\in S$ such that $\pi_n(y)\in V\cap\pi_n(S)$. But then $y\in 
U\cap S$. Since this applies to any neighbourhood of $z$, it follows $z\in\overline S$, proving
$$\overline S\supseteq\bigcap_{n\in\N_1}\pi_n^{-1}\big(\overline{\pi_n(S)}\big)\;.$$
\end{prf}

\begin{cor} \label{c:pjdns} Let $\calt$ be a p\'eij\'i topology on $\Z$. A subset $S$ is dense in $(\Z,\calt)$ if and only if $\pi_n(S)$ is dense in each $(\Z/n\Z,\calt_n)$, with $\{\calt_n\}$ the 
maximal family inducing $\calt$. \end{cor}

Thus, for example, $\N$ and $\N_1$ are dense subsets of $\Z$ with respect to any p\'eij\'i topology $\calt$.

\subsubsection{CRT-compatible topologies} If $n,k\in\N_1$ are coprime, the Chinese remainder theorem yields a ring isomorphism
\begin{equation} \label{e:CRT} \psi_{n,k}\colon\Z/nk\Z\longrightarrow\Z/n\Z\times\Z/k\Z \end{equation}
with the property
$$(\pi_n^{nk})^{-1}(A)\cap(\pi_k^{nk})^{-1}(B)=\psi_{n,k}^{-1}(A\times B)$$
for any $A\subseteq\Z/n\Z$ and $B\subseteq\Z/k\Z$. Thus if the codomain of $\psi_{n,k}$ is given the product topology and $\{\calt_n\}$ is maximal, Corollary \ref{c:fnsttplg} shows that also the map 
$\psi_{n,k}$ is continuous. However, there is no need that $\psi_{n,k}$ is a homeomorphism: look at \eqref{e:egt6} for an example where $\Z/2\Z\times\Z/3\Z$ has the trivial topology, but $\Z/6\Z$ has 
not.

We say that a family $\{\calt_n\}$ is {\em CRT-compatible} if the isomorphisms $\psi_{n,k}$ described by the Chinese remainder theorem are all homeomorphisms with respect to $\{\calt_n\}$. Thus a 
CRT-compatible family is completely determined by the topologies $\calt_{p^j}$, with $p^j$ varying among all prime powers: more precisely, $\{\calt_n\}$ is CRT-compatible if and only if the sets of 
the form
\begin{equation} \label{e:crtopn} U=\bigcap_{p|n}(\pi_{p^{v_p(n)}}^n)^{-1}(U_{p^{v_p(n)}})\,, \end{equation}
with $U_{p^{v_p(n)}}$ open in $\Z/p^{v_p(n)}\Z$, form a base of $\calt_n$ for each $n$.

It is trivial that, when $\{\calt_n\}$ is CRT-compatible, the maps $\pi_n^{nk}$ are continuous and open for any coprime pair $n,k$.\\

The following lemma can be useful to check CRT-compatibility. The proof is an easy exercise and will be left to the reader.

\begin{lem} \label{l:qtprd} Let $(X,\calt_X)$ and $(Y,\calt_Y)$ be two topological spaces and $(X',\calt_X')$ and $(Y',\calt_Y')$ be their quotients for some equivalence relations. Then on $X'\times 
Y'$ the product topology $\calt_X'\times\calt_Y'$ coincides with the quotient of $\calt_X\times\calt_Y$. \end{lem}

\subsubsection{Determining sets} \label{sss:dtrms} In many situations, working with the maximal family $\{\calt_n\}$ attached to $\calt$ can become cumbersome. Let $\M$ be a nonempty subset of 
$\N_1$. We say that {\em $\M$ is a determining set for $\calt$} if $\{\calt_n\}$ is equivalent to $\{\calt_n'\}$ with $\calt_n'=\calt_n$ when $n\in\M$ and $\calt_n'$ the trivial topology otherwise. 
For our purposes, it will often be convenient to take as $\M$ a monoid.

If $\M$ and $\N_1$ are cofinal with respect to the divisibility order (that is, any $n\in\N_1$ has a multiple $m\in\M$), then $\M$ is determining for any p\'eij\'i topology $\calt$. The proof is 
trivial.

\begin{rmk} \label{r:28} By definition of determining set, the sets $\pi_m^{-1}(W)$, with $m\in\M$ and $W\in\calt_m$, form a subbase for $\calt$. The reasoning of Remark \ref{r:mxmfml}.2 above shows 
that this is a base when $\M$ is a monoid (since then one can apply \eqref{e:aprtn} with $n_1,...,n_k\in\M$). 

Note also that a p\'eij\'i topology is uniquely defined by the choice of a family $\{\calt_n\}_{n\in\M}$ with an arbitrary $\M\subseteq\N_1$ taken as determining set. \end{rmk}

\begin{lem} \label{l:ndm} Let $\calt$ be a p\'eij\'i topology on $\Z$ and $\{\calt_n\}$ the maximal family inducing it. Assume $\M$ is a determining monoid for $\calt$. For a fixed $n\in\N_1$, let 
$d$ be the greatest divisor of $n$ which also divides some $m\in\M$. Take $U\subseteq\Z/n\Z$ and put $V=\pi_d^n(U)$. Then $U\in\calt_n$ if and only if $U=(\pi_d^n)^{-1}(V)$ and $V\in\calt_d$. 
\end{lem}

\begin{prf} By Proposition \ref{p:fnsttplg}, for all $k$ we have that $X\in\calt_k$ is equivalent to $\pi_k^{-1}(X)\in\calt$. Thus one implication follows if we can show that $U$ open implies 
$U=(\pi_d^n)^{-1}(V)$; the other one is immediate from the fact that $\pi_d^n$ is continuous.

Take $x=[a]_n\in U$. Write $n=dh\in\N_1$ and note that $d$ and $h$ are coprime by construction. Then \eqref{e:CRT} implies $(\pi_d^n)^{-1}(\pi_d^n(x))=\{x_1,...,x_h\}$, where the $x_i$'s differ by 
their image in $\Z/h\Z$. By Remark \ref{r:28}, if $U$ is open, we have $a\in\pi_m^{-1}(W)\subseteq\pi_n^{-1}(U)$ for some $W\in\calt_m$ and hence $\pi_n(a+m\Z)\subseteq U$. There is no loss of 
generality in assuming that moreover $m$ is a multiple of $d$: but then $\pi_d(a+m\Z)=\pi_d(x)$, while $\pi_h(a+m\Z)=\Z/h\Z$ (because $h$ and $m$ are coprime). Hence the Chinese remainder theorem 
yields $\pi_n(a+m\Z)=\{x_1,...,x_h\}$. \end{prf}

In particular, if $n$ is coprime with all $m\in\M$ then $\calt_n$ is the trivial topology.

\subsubsection{Ring topologies} Let $R$ be a ring endowed with a topology $\calt$. Recall that $(R,\calt)$ is a {\em topological ring} if the ring operations are continuous maps $R\times 
R\rightarrow R$ (with respect to the product topology on $R\times R$). 

\begin{prop} \label{p:tplrng} Assume $\{\calt_n\}$ is maximal. Addition and multiplication are continuous on $(\Z,\calt)$ if and only if they are so on $(\Z/n\Z,\calt_n)$ for every $n$. \end{prop}

\begin{prf} Consider the commutative diagram
\begin{equation} \label{e:cdanlltpl} \begin{CD} \Z\times\Z @>f>> \Z \\
@V(\iota,\iota)VV  @VV{\iota=\prod\pi_n}V \\
\displaystyle\prod_{n\in\N_1}\Z/n\Z\times\prod_{n\in\N_1}\Z/n\Z @>{\prod f_n}>> \displaystyle\prod_{n\in\N_1}\Z/n\Z
\end{CD} \end{equation}
where $f$ and $f_n$ denote either addition or multiplication and all products are given the product topology. We have to show that $f$ is continuous if and only if so are all $f_n$'s.

First assume that each $(\Z/n\Z,\calt_n)$ is a topological ring. Then $\prod f_n$ is continuous and this implies that so is $f$, by definition of $\calt$ and the continuity of $(\iota,\iota)$. 

Now assume that $(\Z,\calt)$ is a topological ring and let $A$ be an open subset of $\Z/n\Z$. Then $f^{-1}(\pi_n^{-1}(A))$ is open and this implies that so is $f_n^{-1}(A)$, because 
\eqref{e:cdanlltpl} is commutative and (as one easily checks by Lemma \ref{l:qtprd}) the product topology on $\Z/n\Z\times\Z/n\Z$ is the same as the one as a quotient of 
$\Z\times\Z$. \end{prf}

\begin{rmk} \label{r:ntmx} From the proof it is clear that the family $\{\calt_n\}$ need not be maximal if one just wants to check that an operation is continuous on $\Z$ (or on one of 
$\N,\N_1$). \end{rmk}

It is easy to see that, for any ring $R$, sum and product are continuous for the topology on $R$ having as base all the cosets of a fixed ideal. In the case of $\Z/n\Z$ this is the only way to 
obtain a ring topology, as shown below.

\begin{lem} \label{l:fnttprng} Any topology making $\Z/n\Z$ into a topological ring has a base of open sets consisting of the cosets of an ideal. \end{lem}

\begin{skprf} By finiteness, the intersection of all open neighbourhoods of $[0]_n$ is still an open set: call it $I$. Continuity implies that multiplying by $[-1]_n$ and adding a fixed $[a]_n$ are 
homeomorphisms of $\Z/n\Z$ into itself. It follows that $I$ is an ideal and all its cosets are open. They form a base because otherwise there would be a neighbourhood of $[0]_n$ strictly smaller 
than $I$. \end{skprf}

Putting together Proposition \ref{p:tplrng} and Lemma \ref{l:fnttprng}, we have a complete description of the p\'eij\'i topologies making $(\Z,\calt)$ a topological ring: each of them corresponds to 
a family $\{\calt_n\}_{n\in S}$ for some $S\subseteq\N_1$ such that $\calt_n$ is the discrete topology on $\Z/n\Z$.

\begin{rmks} \begin{itemize} \item[]\end{itemize}
\noindent{\bf 1.} Usually one puts being Hausdorff among the requirements for a ring topology. In the description above, this means that $0$ must be the only common multiple of all elements in $S$.\\
\noindent{\bf 2.} There are uncountably many Hausdorff ring topologies on $\Z$ for which a base of neighbourhoods of $0$ does not consist of ideals - see e.g. \cite{hinr}. \end{rmks}

\subsection{Examples of p\'eij\'i topologies} \label{ss:egpj} In the following we are going to show how a number of examples one can find in the literature fall under our definition: here and in 
subsection \ref{ss:applct} below we summarize (and often simplify) many previous works, in some cases also pointing out a few easy, but apparently undetected, consequences. Note that in many instances 
the underlying set in the definition we cite was $\N$ or $\N_1$ and not $\Z$; in those cases the original topological space can be recovered by restriction (see Remark \ref{r:NZint} below for 
details). In \S\ref{sss:tplgN} we will say more about the relationship between topologies on $\Z$ and on $\N$.

Because of finiteness, a convenient way of defining a topology on $\Z/n\Z$ is to choose some singletons as open points and take them as a base for $\calt_n$. In our examples, we will usually follow 
this path. Also, as noted in \S\ref{sss:dtrms}, one could define $\calt_n$ only for $n$ in a determining set $\M\subsetneq\N_1$: for an example with $\M=\calp$, see \S\ref{sss:krch} below.

\subsubsection{Broughan topologies} The approach of \cite{brou} is to fix for each $a\in\Z$ a multiplicative submonoid $G_a\subseteq\N_1$ and take the cosets $a+b\Z$, with $b$ varying in $G_a$, 
as a subbase for a topology on $\Z$. \footnote{A somewhat similar approach, with $\Z$ replaced by any integral domain, can be found in \cite[\S2]{mrkprb}.} Broughan calls ``adic'' the topologies 
obtained by this construction, using the symbol $\calt_\calg$ to denote the one attached to $\calg=\{G_a : a\in \Z\}\,.$ In our approach, his definition can be reformulated as the request that the 
point $[a]_n$ is open in $\Z/n\Z$ if $n\in G_a$ and letting $\calt_n$ be the topology defined by this condition. Then one has $\calt_\calg=\calt$, proving that our p\'eij\'i topologies include all of 
Broughan's ones.

In the examples to follow, we will point out if they are Broughan or not.

\begin{rmks} \begin{itemize} \item[] \end{itemize}
\noindent {\bf 1.} The trivial topology is in the scope of Broughan's definition, since it is obtained by taking $G_a=\{1\}$ for all $a$. Extending Broughan's definition to take $G_a=\N$ for every 
$a\in\Z$ yields the discrete topology, which is not among the ones considered either in \cite{brou} or in the present paper. (Note that the convention implicit in \cite{brou} is $0\notin\N$.)\\
{\bf 2.} We point out a mistake in \cite[page 712]{brou}. Just before Definition 2.1, one finds the statement ``All adic topologies make the multiplication $\cdot$ continuous'' (with no proof). 
Actually this is false: for a counterexample, take
$$G_a:=\begin{cases}\{2^n\mid n\in\N\}&\text{ if }a=2\\ \{3^n\mid n\in\N\}&\text{ if }a=3\\ \{5^n\mid n\in\N\}&\text{ if }a=6\\\{1\}&\text{ otherwise. } \end{cases}$$
Then the sequences $(2+2^{5^n})_n$ and $(3+3^{5^n})_n$ converge respectively to $2$ and $3$, but
$$(2+2^{5^n})\cdot(3+3^{5^n})\equiv(2+2)\cdot(3+3)\equiv4\mod5\;\;\;\forall\,n\in\N$$
implies $(2+2^{5^n})\cdot(3+3^{5^n})\notin6+5\Z$ and therefore the sequence of products cannot converge to $6$, showing that multiplication is not continuous in this $\calt_\calg$. \end{rmks}

\subsubsection{The Furstenberg topology} We will refer to it by this name, but as far as we know this topology was already well-known to algebraists before Furstenberg, whose merit lies in realizing 
the application to the infinitude of primes. It is obtained by giving the discrete topology $\calt_{F,n}$ to each quotient $\Z/n\Z$. It is a Broughan topology, taking $G_a=\N_1$ for every $a$. It is 
obvious that the Furstenberg topology is the finest among the Broughan and p\'eij\'i ones: for this reason it is called the {\em full topology} in \cite{brou}; it is also named the {\em linear 
topology} in \cite{kp} and the adic topology in \cite{cllp}. We also note that $(\Z,\calt_F)$ is a totally disconnected topological space (because all cosets of non-zero ideals are both open and 
closed) and it is not compact. An elementary introduction to the space $(\Z,\calt_F)$ can be found in \cite{lm}.

\subsubsection{The topology of non-trivial cosets} \label{sss:tkp} On $\Z/n\Z$ define a topology $\calt_{KP,\,n}$ by taking all singletons different from $\{[0]_n\}$ as open. It is a Broughan 
topology, with 
$$G_a:=\begin{cases}\{1\}&\text{ if }a=0\\ \{n\in \N_1:n\nmid a\} & \text{ if }a\neq0. \end{cases}$$
One sees immediately that the family $\{\calt_{KP,\,n}\}$ is maximal by observing that the only open subset of $\Z$ containing $0$ is $\Z$ itself, while any refinement of a $\calt_{KP,\,n}$ would add 
a neighbourhood of $0$. This observation also implies that $(\Z,\calt_{KP})$ is compact and connected and that $0$ is the only closed point (because any closed subset must contain $0$). Restricting 
to  $\Z-\{0\}$, $\calt_{KP}$ is the same as the Furstenberg topology (in particular, it is totally disconnected).

This example is considered by Knopfmacher and Porubsk\'y in \cite[Proposition 3]{kp} (in a more general setting, with $\Z$ replaced by any integral domain $R$).

\subsubsection{The Golomb topology} It results from the topologies $\calt_{G,n}$ obtained by taking all points in $(\Z/n\Z)^*$ as open singletons. In Broughan's construction, this amounts to taking 
$$G_a=\{n\in \N_1:(a,n)=1 \}$$
for every $a\in\Z$ (in particular, $G_0=\{1\}$).

The map $\pi_n^{nk}$ is not continuous with respect to $\calt_{G,nk}$ if $n$ and $k$ are coprime, because the inverse image of the unit $[k]_n$ contains the zero-divisor $[k]_{nk}$\,. By 
Corollary \ref{c:fnsttplg}, this shows that $\{\calt_{G,n}\}$ is not maximal. Take $\calt_{G,n}'=\calt_{G,n}$ if $n$ is a prime power and define $\calt_{G,n}'$ by \eqref{e:crtopn} otherwise: then 
Lemma \ref{l:mxmcrt} ensures that $\{\calt_{G,n}'\}$ is a maximal family and one easily sees that it is equivalent to $\{\calt_{G,n}\}$. 

The topology $\calt_G$ was first defined by Brown in \cite{brwn} and later independently rediscovered and popularized by Golomb, who also gave some arithmetic applications in his papers \cite{gol1} 
and \cite{gol2} (see \S\ref{sss:glmar} below). Both these authors took $\N_1$ as the underlying set; the extension to $\Z$ (and more general rings) can be found in \cite{kp} and \cite{mrkprb}, where 
it is called the {\em invertible cosets topology}, and in \cite{brou}, with the name of {\em coprime topology}. 

It is immediate from the definition of $\calt_G$ that the only neighbourhood of $0$ is $\Z$ itself, which therefore is connected and compact. With a little more effort, one can show that the 
restriction of $\calt_G$ to $\N_1$ and $\Z-\{0\}$ are connected as well (\cite[Theorem 3]{gol1}; see also the nice presentation in \cite{cllp}). Moreover both these subsets are Hausdorff (since any 
two non-zero numbers can be separated in $\Z/p\Z$ with $p$ a big enough prime) and non-compact (\cite[Theorem 5]{gol1} and \cite[Theorem 2.2]{brou} respectively). Actually, the pair $(\N_1,\calt_G)$ 
was first introduced by Brown as an example of a countable connected Hausdorff space. More recently, the same idea has been extended to other integral domains $R$ so to make $R-\{0\}$ a space with 
similar properties: we refer to \cite{kp}, \cite{cllp} and \cite{sp} for detailed studies of such topologies. See also \cite{bmt} and \cite{bspt} for a deeper analysis of $(\N_1,\calt_G)$.

Multiplication is a continuous operation under $\calt_G$\,. This is easily checked by Proposition \ref{p:tplrng} and was already noticed in \cite[Proposition 8]{kp} and in \cite[Lemma 2.6]{orum}.

\subsubsection{Kirch's topology} \label{sss:krch} Put $\F_p:=\Z/p\Z$. This topology (first defined in \cite{krch} on $\N_1$) can be better understood as induced by the embedding
$$\Z\hooklongrightarrow\prod_{p\in\calp}\F_p$$
where for each $p$ all points of $\F_p^*$ are open and $[0]_p$ is not. That is, we take $\calt_{K,n}=\calt_{G,n}$ if $n=p$ is prime and $\calt_{K,n}=\{\Z/n\Z,\emptyset\}$ otherwise. (For the finest 
family inducing $\calt_K$, when $n$ is not prime let $\calt_{K,n}'$ be the coarsest topology such that the maps $\pi_p^n$ are continuous for every prime $p$ dividing $n$.) This is not a Broughan 
topology (because if $p$ is in $G_a$ then so must be $p^2$).

It is clear that $\calt_K$ is a coarsening of $\calt_G$ (and so $\Z$, $\N_1$ and $\Z-\{0\}$ are still connected under $\calt_K$). The same proof used for $\calt_G$ shows that $\N_1$ and 
$\Z-\{0\}$ are Hausdorff under $\calt_K$ as well. Furthermore, $(\N_1,\calt_G)$ is not locally connected, but $(\N_1,\calt_K)$ is (\cite[Theorems 1 and 5]{krch}). 

As in the case of $\calt_G$, multiplication is continuous under $\calt_K$ - a fact first proved in \cite[\S3.1]{orum}. A deeper topological study of $(\N_1,\calt_K)$ is the subject of \cite{bst}.

\begin{rmk} Readers can have fun in considering different ways of defining analogues of $\calt_K$ with $\Z$ replaced by a more general $R$. The first obvious (and non-equivalent) ideas are to use the 
map $R\rightarrow\prod_{I\in\cals}R/I$, with $\cals$ either the set of maximal ideals or of non-zero prime ideals of $R$. In the case of a semiprimitive Dedekind domain mentioned in 
\cite[\S4.1]{cllp} the two methods coincide and yield the same result as the construction suggested there. \end{rmk}

\subsubsection{Rizza's division topology} \label{sss:rzz} For each $n\in\N_1$ take
$$\calt_{R,n}:=\big\{\emptyset,\{[0]_n\},\Z/n\Z\big\}$$ 
(i.e., the topology where the only non-trivial open set is $\{[0]_n\}$). Then the nonempty open sets in $\calt_R$ are the unions of non-trivial ideals of $\Z$. A base of neighbourhoods at a point $a$ 
are the ideals containing $a$: it follows that for $X\subseteq\Z$ and $a\neq0$ one has 
\begin{equation} \label{e:clsrR} a\in\overline{X}^{\calt_R}\,\Longleftrightarrow\;\exists\,b\in X\text{ such that }a\text{ divides }b\,, \end{equation}
which is almost the definition Rizza gave in \cite[\S5]{rzz} (see Remark \ref{r:rzz} below).

The maximal family inducing $\calt_R$ is given by $\calt_{R,n}'$ consisting of $\emptyset$ and the unions of ideals of $\Z/n\Z$. Once again, it is a Broughan topology:  e.g., take 
$$G_a=\begin{cases} \N_1&\text{ if }a=0;\\ \{1\}&\text{ if }a\neq0.\end{cases}$$

Among topological properties of $(\Z,\calt_R)$, we note that every nonempty open set contains $0$. It follows that open sets are all connected and all closed sets but $\Z$ have empty interior. Also, 
by \eqref{e:clsrR}, every nonempty closed set contains $\{\pm1\}$, implying that $(\Z,\calt_R)$ and $(\N_1\,,\calt_R)$ are compact. Characterizations of connected and compact subspaces of 
$(\N_1\,,\calt_R)$ can be found in \cite[Theorem 7.3]{szc16} and \cite[Theorem 7.4]{szc16}.

\begin{rmk} \label{r:rzz} The definition in \cite[\S5]{rzz} implies that \eqref{e:clsrR} holds also for $b=0$ and hence that the singleton $\{0\}$ is an open set. Our approach could provide open 
points only after changing the set of indexes from $\N_1$ to $\N$; we refrain from doing so because of convenience for comparison with the other topologies we are interested in. The only difference 
between what we call ``Rizza's topology'' and the one effectively considered in \cite{rzz} is that the latter has $\{0\}$ among the open sets; in particular, restricting to $\N_1$ they are the same.

From our point of view, Rizza's topology arises from the family $\{\calt_{R,n}\}$. However, it can also be seen as originated by the preorder induced on $\Z$ by divisibility (for generalities on 
topologies coming from a preorder, see \cite[\S2]{hauk}). More examples of topologies on $\Z$ or $\N_1$ (and more general semigroups) obtained from divisibility relations can be found in \cite{prb00} 
and \cite{hauk}, together with some arithmetical applications.

We will meet again order-induced topologies in \S\ref{sss:ordtplg}. \end{rmk}

\subsubsection{Szczuka's common division topology} This topology (named after its appearance in \cite{szc13} for $\N_1$) is a refinement of $\calt_R$ defined by taking as a base of open points for 
$\calt_{S,n}$ the nilpotent elements of $\Z/n\Z$ (that is, those $x\in\Z/n\Z$ such that $x^r=[0]_n$ for some $r\in\N_1)$. Since $[a]_n$ is nilpotent if and only if $\supp(n)\subseteq\supp(a)$, our 
definition is equivalent to the original one in \cite[\S3]{szc13}. The family $\{\calt_{S,n}\}$ is not maximal, because the maps $\pi_n^{kn}$ are not continuous for $n$ and $k$ coprime (e.g., 
$[10]_{12}$ is not nilpotent, but its image in $\Z/4\Z$ is). As with $\calt_G$, maximality is achieved by taking $\calt_{S,n}'=\calt_{S,n}$ if $n$ is a prime power and defining $\calt_{S,n}'$ by 
CRT-compatibility otherwise.\\ 

Also $\calt_S$ is a Broughan topology: actually, \cite[Example 2.3]{brou} provides three different constructions all yielding $\calt_S$, namely letting $G_a$, for $a\neq0$, be generated by the prime 
divisors of $a$, by the maximal prime powers dividing $a$ or by the powers of $a$.

Since $\pm[1]_n$ cannot be nilpotent, the numbers $\pm1$ are contained in every nonempty closed subset with respect to $\calt_S$\,. It follows that $(\Z,\calt_S)$ and $(\N_1,\calt_S)$ are connected 
and compact (\cite[Proposition 3.2]{szc13}), but not Hausdorff. Besides, one can deduce that $1$ is fixed by any continuous, non-constant function from $(\N_1,\calt_S)$ to itself and therefore 
neither sum nor product are continuous with respect to $\calt_S$ (\cite[Theorem 3.2 and Remark 3.4]{szsz}). Finally, we note that $(\N_1,\calt_S)$ is not locally connected (\cite[Corollary 
3.5]{szc13}).

\begin{rmk} \label{r:NZint} As mentioned, the topologies $\calt_G$, $\calt_K$ and $\calt_S$ were originally considered on $\N_1$ and not on $\Z$. Moreover, the definitions one can find in the 
literature take a base consisting of arithmetic progressions and not of cosets: since the equality $(a+n\Z)\cap\N_1=a+n\N$ holds only if $0<a\le n$, at first sight it might appear that our spaces 
$(\N_1,\calt_\bullet)$ have less open sets than the original ones. To see it is not the case, observe that if $n<a\le bn$ one can find a strictly increasing sequence $(s_i)_{i\in\N}$ of positive 
integers such that $s_0\ge b$ and $(a+in)+s_in\Z$ is in $\calt_\bullet$ (e.g., take $s_i\in\calp$ for $\calt_G$ and $\calt_K$ and $s_i$ a sufficiently big power of the product of all the primes in 
$\supp(a+in)$ for $\calt_S$). Then we have $a+in\le s_in$ for each $i$ and the equality
$$a+n\N=\bigcup_{i=0}^\infty\big(a+in+s_in\N\big)=\N_1\cap\bigcup_{i=0}^\infty\big(a+in+s_in\Z\big)$$
shows that $a+n\N$ is indeed open in $\calt_\bullet|_{\N_1}$. (This proof adapts the one for $\calt_G$ in \cite[Proposition 2.1]{orum}.) \end{rmk}

\subsubsection{The topology of zero-divisors} \label{sss:tzd} The fourth topology mentioned in \cite[Example 2.3]{brou} takes as $G_a$ the monoid generated by all multiples of $a$ in $\N_1$, for 
$a\neq0$. Hence the singleton $\{[a]_n\}$ is open if and only if $n=ab$ for some $b\in\Z$, that is, if and only if $[a]_n$ is a zero-divisor. Thus this topology is obtained letting $\calt_{zd,n}$ 
consist of $\Z/n\Z$ and the powerset of the zero-divisors. All maps $\pi_n^{nk}$ are continuous, but they are not open if $k$ and $n$ are coprime.

Note that if $|\supp(n)|>1$ then the topology $\calt_{zd,n}$ is strictly finer than $\calt_{S,n}'$: since the latter form a maximal family (by Lemma \ref{l:mxmcrt}), it follows that $\calt_{zd}$ and 
$\calt_S$ are different. As with $\calt_S$, the space $(\Z,\calt_{zd})$ is connected and compact because every closed set contains $\{\pm1\}$. However, things change dramatically when one takes these 
two points away.

\begin{prop} \label{p:tzd} The topological space $(\Z-\{\pm1\},\calt_{zd})$ is totally disconnected. \end{prop}

\begin{prf} We are going to show the slightly stronger result that the space is totally separated - that is, that given any two distinct elements $a,b$ in $\Z-\{\pm1\}$, the latter can be covered by 
disjoint open sets containing respectively $a$ and $b$. Take $n>|a|+|b|$. Then $[a]_{n!}$ and $[b]_{n!}$ are two distinct zero-divisors in $\Z/n!\Z$. Consider
$$U=\pi_{n!}^{-1}\big([a]_{n!}\big)$$ 
and $V=\bigcup_{k\ge n}\pi_{k!}^{-1}(V_k)$, with
$$V_k=\big\{x\in\Z/k!\Z\mid x\text{ is a zero-divisor and }\pi_{n!}^{k!}(x)\neq[a]_{n!}\big\}\,.$$
Then $U$ and $V$ are disjoint open sets covering $\Z-\{\pm1\}$ with $a\in U$ and $b\in V$. \end{prf}
 
\subsubsection{Broughan's $m$-topology} This is \cite[Example 2.4]{brou}. Broughan's approach is to fix $m\in\N_1$ and define $G_a$ as consisting of all those $b\in\Z$ such that $m$ is nilpotent 
modulo the greatest common divisor of $a$ and $b$, i.e.,
$$b\in G_a\Longleftrightarrow\supp(a)\cap\supp(b)\subseteq\supp(m)\,.$$
Alternatively, this topology can be described by putting
$$\calt_{p^r}:=\begin{cases} \calt_{F,p^r} & \text{ if }p\in\supp(m) \\ \calt_{G,p^r} & \text{ if }p\notin\supp(m)\end{cases}$$
for prime powers and then using the CRT-rule \eqref{e:crtopn} for $\calt_n$ when $n$ has more than one prime divisor. Checking that the two definitions are equivalent is easy and will be left to the 
reader.

The $m$-topology is thus a mixture of the ones of Furstenberg and Golomb. It is not hard to show that it makes $\Z$ into a totally disconnected, non-compact topological space.

\subsection{A family of connected topologies}  \label{ss:BG} As in \cite[\S2]{bmt} and \cite[\S2]{bst}, we say that a topological space $X$ is {\em superconnected} if for any $r\in\N_1$ and 
nonempty open subsets $U_1,\dots,U_r$ of $X$, we have $\bigcap_i\overline{U_i}\neq\emptyset$. If this property holds for $r=2$, we say that $X$ is a {\em Brown space} (following \cite[page 
77]{cllp} - note that superconnected spaces are called ``strongly Brown spaces'' in \cite[page 82]{cllp}). The implications
$$\text{superconnected }\Longrightarrow\text{ Brown }\Longrightarrow\text{ connected}$$
are obvious.

\subsubsection{P\'eij\'i topologies of Brown type} \label{sss:brwntyp} Let $B_n=B_n(\calt_n)$ be the intersection of all nonempty closed subsets for a topology $\calt_n$ on $\Z/n\Z$. One sees 
immediately that $B_n$ is the biggest subset of $\Z/n\Z$ such that
\begin{equation} \label{e:brwn0} B_n\cap U=\emptyset\text{ for any proper open subset }U\text{ of }(\Z/n\Z,\calt_n)\,.  \end{equation}

\begin{lem} \label{l:bnkn} If $\calt_n$ and $\calt_{nk}$ are such that $\pi_n^{nk}$ is continuous then one has $\pi_n^{nk}(B_{nk})\subseteq B_n$. \end{lem}

\begin{prf} Given $x\in B_{nk}$, let $U$ be an open set of $\calt_n$ containing $\pi_n^{nk}(x)$. Then $V=(\pi_n^{nk})^{-1}(U)$ is an open set containing $x$. From $V\cap B_{nk}\neq\emptyset$ and 
\eqref{e:brwn0} it follows $V=\Z/nk\Z$, implying $U=\Z/n\Z$. \end{prf}

\begin{dfn} \label{d:brwntyp} Let $\calt$ be a p\'eij\'i topology and $\{\calt_n\}$ the corresponding maximal family. We say that $\calt$ is {\em of Brown type} if it admits a determining 
monoid $\M\subseteq\N_1$ such that, for every $n,m\in\M$, with $m$ a multiple of $n$, \begin{itemize}
\item[{\bf B1}] the map $\pi_n^m$ is open;
\item[{\bf B2}] $\pi_n^m(B_m)=B_n$ \end{itemize}
(where the topology on $\Z/m\Z$ is $\calt_m$ for all $m\in\N_1$). \end{dfn} 

Note that $B_1$ is never empty (because $\calt_1$ is the unique topology which can be assigned to a singleton). Since $\M$ contains $1$, condition {\bf B2} implies $B_n\neq\emptyset$ for every 
$n\in\M$. Readers can easily check that the two requirements {\bf B1} and {\bf B2} are independent of each other.

\begin{eg} \label{eg:brwntyp} The following p\'eij\'i topologies are of Brown type (with $\M=\N_1$ in all cases): \begin{itemize} 
\item the Golomb topology $\calt_G$\,, with $B_n$ the set of nilpotent elements in $\Z/n\Z$;
\item the Kirch topology $\calt_K$\,, with $B_n$ the set of nilpotent elements in $\Z/n\Z$;
\item the Rizza topology $\calt_R$\,, with $B_n$ the set of units in $\Z/n\Z$;
\item the Szczuka topology $\calt_S$\,, with $B_n$ the set of units in $\Z/n\Z$. \end{itemize} \end{eg}

\begin{lem} \label{l:brwntyp} Let $\calt$, $\{\calt_n\}$ and $\M$ be as in Definition \ref{d:brwntyp}. If {\bf B1} holds, then $n\in\M$ implies that the equality 
$$\overline{\pi_n^{-1}(X)}^\calt=\pi_n^{-1}\big(\overline X^{\calt_n}\big)$$ 
is true for every $X\subseteq\Z/n\Z$. \end{lem}

\begin{prf} If $X=\emptyset$, there is nothing to prove. So fix $X\neq\emptyset$ and take $a\in\Z$ such that $[a]_n\in\overline X^{\calt_n}$. Let $U$ be a neighbourhood of $a$. By Remark 
\ref{r:28}, there are $m\in\M$ and $W$ an open neighbourhood of $[a]_m$ in $\calt_m$ such that $U$ contains $\pi_m^{-1}(W)$. At the cost of replacing $W$ with $(\pi_m^{nm})^{-1}(W)$, we can 
assume that $n$ divides $m$. Then condition {\bf B1} implies that $\pi_n^m(W)$ is an open neighbourhood of $[a]_n$ and, as such, it must intersect $X$. This shows 
$U\cap\pi_n^{-1}(X)\neq\emptyset$ and, because $a$ and $U$ were arbitrary, yields the inclusion
$$\pi_n^{-1}\big(\overline X^{\calt_n}\big)\subseteq\overline{\pi_n^{-1}(X)}^\calt\,.$$ 
The opposite inclusion is trivial. \end{prf}

In the literature one can find many computations of the closure in $\N_1$ of arithmetic progressions with respect to $\calt_G$, $\calt_K$, $\calt_R$ and $\calt_S$: see for example the papers  
\cite{szc14b} and \cite{szc15}, as well as \cite[Lemma 2.2]{bmt}, \cite[Lemma 1]{bst}, \cite[Lemmata 3.5, 3.8 and 3.9]{szc14a} and \cite[\S6]{szc16}. Since $a+b\N$ is, up to a finite subset, the same 
as  $\pi_b^{-1}([a]_b)\cap\N$, it is very easy to recover all these results from Lemma \ref{l:brwntyp}. We leave details to interested readers.

\begin{cor} \label{c:brwntyp} In the conditions of Lemma \ref{l:brwntyp}, assume $[b]_m\in B_m$ for some $b\in\Z$ and $m\in\M$. Then $b$ is in the closure of $\pi_m^{-1}(X)$ for any nonempty 
$X\subseteq\Z/m\Z$.  \end{cor}

\begin{prf} Obvious, since \eqref{e:brwn0} implies that $\{[b]_m\}$ is dense in $\Z/m\Z$. Actually, the proof of Lemma \ref{l:brwntyp} shows that $\pi_m(W)=\Z/m\Z$ holds for every 
$\calt$-neighbourhood $W$ of $b$. \end{prf}

\begin{thm} \label{t:brwntyp} Let $\calt$ be a p\'eij\'i topology of Brown type, with $\{\calt_n\}$ and $\M$ as in Definition \ref{d:brwntyp}. Assume $[a]_n\in B_n$ for some $a\in\Z$ and $n\in\M$. 
Then $a+n\Z$ and $a+n\N$ are superconnected. In particular, if $\calt$ is of Brown type then $(\Z,\calt)$, $(\N,\calt)$ and $(\N_1,\calt)$ are superconnected.  \end{thm}

\begin{prf} The last statement in the theorem follows immediately from the previous one by taking $n=1$ and $a$ either $0$ or $1$. 

Let $A$ denote either $a+n\Z$ or $a+n\N$. Fix $r\in\N_1$ and let $U_1,\dots,U_r$ be open subsets of $(\Z,\calt)$ such that each $A_i=A\cap U_i$ is nonempty: we are going to show that then
$\cap_i\overline{A_i}$ contains a point of $A$. 

Without loss of generality, we can assume that for each $i$ there are $m_i\in\M$ and $V_i$ open in $\calt_{m_i}$ so that $\pi_{m_i}^{-1}(V_i)=U_i$. Let 
$k$ be the product of $m_1,\dots,m_r$ and take
$$X_i=(\pi_{m_i}^{nk})^{-1}(V_i)\cap(\pi_n^{nk})^{-1}([a]_n)\,.$$
Then $m=nk$ is in $\M$ and the sets $X_i=\pi_m(A_i)$ are nonempty. By {\bf B2}, there is $b\in\Z$ such that 
\begin{equation} \label{e:absprcnn} \pi_n^m([b]_m)=[a]_n\;\text{ and }\,\;[b]_m\in B_m\,. \end{equation}
Since \eqref{e:absprcnn} defines $b$ only modulo $nk$, we can assume $b\in A$.  Corollary \ref{c:brwntyp} shows that $b$ is in the closure of each $\pi_m^{-1}(X_i)$. If $A$ is $a+n\Z$ then we 
conclude by $A_i=\pi_m^{-1}(X_i)$. In the case of $a+n\N$, we still have 
$$\overline{A_i}=\overline{\pi_m^{-1}(X_i)}$$
because any open set $W$ in $(\Z,\calt)$ is a union of cosets and so $W\cap\pi_m^{-1}(X_i)\neq\emptyset$ implies $W\cap A_i\neq\emptyset$. \end{prf} 

\begin{rmk} \label{r:cnndscnn} It might be worthwhile to note that Theorem \ref{t:brwntyp} does not apply to $\calt_{KP}$ of \S\ref{sss:tkp} and $\calt_{zd}$ of \S\ref{sss:tzd}, which are not of 
Brown type and which both make $\Z$ connected, but with plenty of totally disconnected subsets. We will come back on this in Example \ref{eg:b2}. \end{rmk}

\subsubsection{Golomb systems} \label{sss:BGS} As the observant reader may have noticed, the topologies in Example \ref{eg:brwntyp} share much more structure than just being  of Brown type. The 
following construction aims at capturing some of this structure. We start by choosing sets $B_n\subseteq\Z/n\Z$ and then define $\calt_n$ so to make this notation compatible with the one used in 
\S\ref{sss:brwntyp} (see Remark \ref{r:bncm}.2 below).\\

Fix a function $\gamma\colon\calp\rightarrow\N$. Then $\gamma$ defines a monoid 
\begin{equation} \label{e:glmsyst1} \M=\M(\gamma):=\big\{n\in\N_1:\gamma(p)\le v_p(n)\;\forall\,p\in\supp(n)\big\}\cup\big\{1\big\}\,.  \end{equation}
For each prime $p$, fix an arbitrary nonempty subset $B_{p^{\gamma(p)}}$ of $\Z/p^{\gamma(p)}\Z$. If $\gamma(p)>0$ we also require $B_{p^{\gamma(p)}}\neq\Z/p^{\gamma(p)}\Z$. The sets $B_n$ are then 
defined for each $n\in\N_1$ by the rules: $B_1\neq\emptyset$,
\begin{equation} \label{e:glmsyst2} B_{p^r}:=(\pi_{p^{\gamma(p)}}^{p^r})^{-1}(B_{p^{\gamma(p)}})  \end{equation}
for $p\in\calp$ and $r\ge\gamma(p)$,
\begin{equation} \label{e:glmsyst3} B_n:=\bigcap_{p|n}(\pi_{p^{v_p(n)}}^n)^{-1}(B_{p^{v_p(n)}}) \end{equation}
when $n\in\M$ is not a prime power and
\begin{equation} \label{e:glmsyst4} B_n:=\pi_n^{\tilde n}(B_{\tilde n})\;\text{ with }\tilde n=\prod_{p|n}p^{\max\{v_p(n),\gamma(p)\}} \end{equation}
for $n\notin\M$.

\begin{rmk} \label{r:gmmB} We started with $\gamma$ only for expository reasons. It is clear from the above that one would obtain the same $B_n$'s replacing $\gamma(p)$ with the smallest $r$ such 
that $|B_{p^r}|=\frac{1}{p}|B_{p^{r+1}}|$ and in the following we will always assume that $\gamma$ is determined by $\calb:=\{B_n\}_{n\in\N_1}$ following this rule. \end{rmk}

We will call any $\calb=\{B_n\}_{n\in\N_1}$ as above a {\em Golomb system}. Given such a $\calb$, a {\em Kirch function} for $\calb$ is a map $\kappa\colon\calp\rightarrow\N\cup\{\infty\}$, with 
$\gamma(p)\le\kappa(p)$ for every $p$.

\begin{dfn} \label{d:glmsyst} To any pair $(\calb,\kappa)$ consisting of a Golomb system and a Kirch function for it, we attach the topology $\calt_{\calb,\kappa}$ having as subbase the cosets
\begin{equation} \label{e:glmbbs} \big\{a+p^r\Z:p\in\calp,\,\gamma(p)\le r\le \kappa(p),\,[a]_{p^r}\notin B_{p^r}\big\}. \end{equation}
\end{dfn}

\begin{eg} \label{eg:glmb} Most of the topologies in Example \ref{eg:brwntyp} fall under Definition \ref{d:glmsyst}. More precisely, the sets of units (for $\calt_S$ and $\calt_R$) and of nilpotents 
(for $\calt_G$ and $\calt_K$) form Golomb systems, with $\gamma$ the constant function $\gamma(p)=1$ in all cases. The Kirch function is $\kappa=\gamma$ for $\calt_K$ and $\kappa(p)=\infty$ for every 
$p$ in the other settings. With these conventions, one sees that $\calt_G$, $\calt_K$ and $\calt_S$ are of the form $\calt_{\calb,\kappa}$\,. However $\calt_R$ is not, because in its case only very 
few of the subsets in the complement of $B_{p^r}$ are open. \footnote{One more example appeared in the literature after this paper was submitted: the topologies $\cald_m$ considered in \cite{sz22} 
are 
the  restriction to $\N_1$ of $\calt_{\calb,\calk}$ with $\{B_n\}$ the sets of nilpotents and $\kappa(p)=m$ for every $p$.} \end{eg} 

\begin{rmks} \label{r:bncm} The following observations apply to any $\calt_{\calb,\kappa}$ as in Definition \ref{d:glmsyst}.\\
\noindent{\bf 1.} It should be obvious that $\calt_{\calb,\kappa}$ is a p\'eij\'i topology. Taking the set $\M$ of \eqref{e:glmsyst1} as a determining monoid, $\calt_{\calb,\kappa}$ can be seen as 
originated by the family $\{\calt_n\}_{n\in\M}$ defined by: \begin{itemize}
\item if $\supp(n)=\{p\}$ is a singleton and $v_p(n)\le\kappa(p)$, then all points outside $B_n$ are open;
\item otherwise, for $n\neq1$, a set $U\subseteq\Z/n\Z$ is open if and only if
\begin{equation} \label{e:glmint} U=\bigcap_{p|n}(\pi_{p^{r(p)}}^n)^{-1}(U_{p^{r(p)}})  \end{equation}
with $r(p)=\min\{v_p(n),\kappa(p)\}$ and $U_{p^{r(p)}}\in\calt_{p^{r(p)}}$ for every $p$. \end{itemize}
These are exactly the quotient topologies defined by $\calt_{\calb,\kappa}$; taking the quotient topology on $\Z/n\Z$ also for $n\notin\M$, we obtain the maximal family.\\
\noindent{\bf 2.} From the description of $\calt_n$ just above, it is clear that $B_n$ is the intersection of all closed subsets in $\Z/n\Z$ for $n\in\M$. In order to see that this is 
true also for $n\notin\M$, just note that, taking $S\subsetneq\Z/n\Z$ and choosing $\tilde n$ as in \eqref{e:glmsyst4}, the set 
$$\pi_n^{-1}(S)=\pi_{\tilde n}^{-1}\big((\pi_n^{\tilde n})^{-1}(S)\big)$$ 
is open under $\calt_{\calb,\kappa}$ if and only if $S\cap B_n=\emptyset$. Thus the notation is coherent with the one used in \S\ref{sss:brwntyp}.\\
\noindent{\bf 3.} Consider
\begin{equation} \label{e:dfMk} \M_\kappa=\M(\calb,\kappa):=\big\{n\in\N_1:\gamma(p)\le v_p(n)\le\kappa(p)\;\forall\,p\in\supp(n)\big\}\,. \end{equation}
Then $\M_\kappa$ is a determining set for $\calt_{\calb,\kappa}$\,.\\
\noindent{\bf 4.} The maximal family $\{\calt_n\}$ described above is CRT-compatible. Since \eqref{e:glmint} and \eqref{e:crtopn} are the same, one has just to check what happens for $n\notin\M$. In 
this case $\calt_n$ is the quotient topology of $\calt_{\tilde n}$, with $\tilde n$ as in \eqref{e:glmsyst4}, which in turn is the product of $\calt_{p^{v_p(\tilde n)}}$ for $p\in\supp(n)$. One 
concludes by the compatibility of quotient and product (Lemma \ref{l:qtprd}).\\
\noindent{\bf 5.} Equality \eqref{e:glmsyst3} holds also for $n\notin\M$. The easiest way to deduce it from \eqref{e:glmsyst4} is probably to remember that, by the Chinese remainder theorem, 
there is a commutative diagram
$$\begin{CD} \Z/\tilde n\Z @<\sim<{\tilde\psi}< \prod_{p|n}\Z/p^{v_p(\tilde n)}\Z\\
@V{\pi_n^{\tilde n}}VV @VV{\prod\pi_{p^{v_p(n)}}^{p^{v_p(\tilde n)}}}V  \\
\Z/n\Z @<\sim<\psi< \prod_{p|n}\Z/p^{v_p(n)}\Z \end{CD}$$ 
(where horizontal maps are isomorphisms) and the right-hand side of \eqref{e:glmsyst3} is $\psi(\prod_{p|n}B_{p^{v_p(n)}})$.\\
\noindent{\bf 6.} If $\gamma(p)=0$ then \eqref{e:glmsyst2} implies $B_{p^r}=\Z/p^r\Z$ for every $r$, so cosets relative to powers of the prime $p$ play no role in defining the topology. Because of 
this, we will always implicitly assume $\kappa(p)=0$ if $\gamma(p)=0$. \end{rmks}

For $a\in\Z$, put
\begin{equation} \label{e:dfSa} S_a=S_a(\calb):=\big\{p\in\calp:[a]_{p^{\gamma(p)}}\notin B_{p^{\gamma(p)}}\big\} \end{equation}
and define $\sigma_a\colon S_a\rightarrow\N$ by
$$\sigma_a(p):=\min\big\{r\in\N:[a]_{p^r}\notin B_{p^r}\big\}\,.$$
Let $G_a$ be the monoid generated by $\{p^r:p\in S_a,\,r\ge\sigma_a(p)\}$. Then \eqref{e:glmsyst2} and \eqref{e:glmsyst4} imply
\begin{equation} \label{e:nGass} n\in G_a\Longleftrightarrow[a]_{p^{v_p(n)}}\notin B_{p^{v_p(n)}}\;\forall\,p\in\supp(n)\,. \end{equation}

\begin{lem} \label{l:nGab} Given integers $a,b,n$, the congruence $a\equiv b$ mod $n$ implies
$$n\in G_a\Longleftrightarrow n\in G_b\,.$$
\end{lem}

\begin{prf} Obvious from \eqref{e:nGass} and the Chinese remainder theorem. \end{prf}

Put
\begin{equation} \label{e:dfGak} G_{a,\kappa}=G_{a,\kappa}(\calb):=\big\{n\in G_a: v_p(n)\le\kappa(p)\;\forall\,p\in\supp(n)\big\}\cup\big\{1\big\}\,. \end{equation}

\begin{lem} \label{l:opntbk} The coset $a+n\Z$ is open in $\calt_{\calb,\kappa}$ if and only if $n$ is in $G_{a,\kappa}$\,\!. \end{lem}

\begin{prf} Because of CRT-compatibility (Remark \ref{r:bncm}.4), it is enough to check the case of $n=p^r$ a prime power. By Definition \ref{d:glmsyst}, $a+p^r\Z$ is open 
for $r$ in the interval $[\gamma(p),\kappa(p)]$, so we only have to deal with the case $\sigma_a(p)\le 
r<\gamma(p)$. For $r$ in this range, we can write  
\begin{equation} \label{e:opntbk} a+p^r\Z=\coprod\big(a_i+p^{\gamma(p)}\Z\big) \end{equation}
and conclude by observing that $p^r\in G_{a,\kappa}$ is the same as $[a]_{p^r}\notin B_{p^r}$, which, by \eqref{e:glmsyst4}, is true if and only if $b\equiv a$ mod $p^r$ implies 
$[b]_{p^{\gamma(p)}}\notin B_{p^{\gamma(p)}}$ - that is, if and only if all cosets on the right-hand side of \eqref{e:opntbk} are in the set \eqref{e:glmbbs}. \end{prf}

\begin{prop} For any pair $(\calb,\kappa)$ as above, the topology $\calt_{\calb,\kappa}$ is of Brown type. It is a Broughan topology if and only if $\kappa(p)=\infty$ holds for every $p$ such that 
$\gamma(p)\neq0$. \end{prop}

\begin{prf} It is clear from Remark \ref{r:bncm}.1 that the maps $\pi_n^m$ are open for all $n,m\in\M$ with $n$ dividing $m$. Remark \ref{r:bncm}.2 implies that {\bf B2} is also true. Thus we see 
that $\calt_{\calb,\kappa}$ is of Brown type. 

The second statement is obvious from Lemma \ref{l:opntbk}. \end{prf}

\begin{rmk} \label{r:ntz} We will simply write $\calt_\calb$ to denote the finest topology attached to the Golomb system $\calb$, that is, the one with $\kappa(p)=\infty$ for every $p$, and refer to 
$\calt_{\calb,\kappa}$ as its Kirch coarsening of level $\kappa$. Clearly, the coarsest topology arising from $\calb$ is obtained by putting $\kappa=\gamma$. \end{rmk}

\subsubsection{Connectedness and separability of Golomb system topologies} \label{sss:GsT} Theorem \ref{t:brwntyp} can be slightly strengthened for topologies arising from a Golomb system.

\begin{prop} \label{p:cnncsq} Let $\calb=\{B_n\}_{n\in\N_1}$ be a Golomb system. Assume $[a]_n\in B_n$ for some $n\in\N_1$ and $a\in\Z$. Then $a+n\Z$ and $a+n\N$ are superconnected with respect to 
$\calt_\calb$ (and hence to $\calt_{\calb,\kappa}$ for every Kirch function $\kappa$). \end{prop}

\begin{prf} Given $U_1,\dots,U_r$ as in the proof of Theorem \ref{t:brwntyp}, one can use the same reasoning as there. The key point is to observe that in the case of $\calt_\calb$, condition 
{\bf B2} holds for every $n,m\in\N_1$ (and not only in $\M)$, as long as $n$ divides $m$. Thus \eqref{e:absprcnn} has always a solution $b$. Moreover, there is no obstruction to apply Corollary 
\ref{c:brwntyp}, because any $n\in\N_1$ has some multiple $m\in\M$.  \end{prf} 

\begin{eg} It might be instructive to see an instance of topology of Brown type for which $[a]_n\in B_n$ does not ensure connectedness. Define the p\'eij\'i topology $\calt$ by \begin{itemize}
\item $\calt_3=\big\{\emptyset,\{[2]_3\},\Z/3\Z\big\}$;\vspace{2pt}
\item if $n$ is an odd multiple of $3$ then a subbase of $\calt_n$ consists of those singletons $\{x\}$ such that $\pi^n_3(x)=[2]_3$\,;\vspace{2pt}
\item $\calt_6$ is the topology with subbase $\big\{\{[1]_6,[3]_6\},\{[4]_6\},\{[2]_6\},\{[5]_6\}\big\}$;\vspace{2pt}
\item if $n$ is a multiple of $6$ then a subbase of $\calt_n$ consists of $(\pi_6^n)^{-1}\big(\{[1]_6,[3]_6\}\big)$ and those singletons $\{x\}$ such that $\{\pi^n_6(x)\}$ is in 
$\calt_6$\,;\vspace{2pt}
\item $\calt_n$ is trivial in the remaining cases.
\end{itemize}
It is clear that $3\N_1$ and $6\N_1$ are determining semigroups for $\calt$. Let $\{\calt_n'\}$ be the maximal family attached to $\calt$. Then the equality $\calt_n'=\calt_n$ is obvious for 
$n\in6\N_1$ and one can check it holds also for odd multiples of $3$. In particular, we have $\calt_3'=\calt_3$. Taking $\{1\}\cup6\N_1$ as $\M$, one sees that $\calt$ is of Brown type. The smallest 
nonempty closed subsets in $\calt_3$ and $\calt_6$ are respectively $\{[0]_3,[1]_3\}$ and $\{[0]_6\}$, showing $B_3\neq\pi_3^6(B_6)$. The coset $1+3\Z$ is not connected with respect to $\calt$, 
because it is covered by the disjoint open sets $4+6\Z$ and $(1+6\Z)\cup(3+6\Z)$. \end{eg}

Our next result explains how Kirch's construction of a locally connected topology fits in the framework of Golomb systems.

\begin{thm} \label{t:glmkrc} The spaces $(\Z,\calt_{\calb,\kappa})$ and $(\N_1,\calt_{\calb,\kappa})$ are locally connected if and only if $\kappa(p)<\infty$ holds for every prime $p$. Moreover, if 
this happens then they are locally superconnected. \end{thm}

\begin{prf} For simplicity we only discuss the case of $\Z$, leaving to readers the easy task of checking that everything works also restricting to $\N_1$.

To see the necessity, fix a prime $p$ satisfying $\kappa(p)=\infty$ and take $a\in\Z$ such that $[a]_p$ is outside $B_p$. Then $\sigma_a(p)=1$ implies that the set $G_{a,\kappa}$ of \eqref{e:dfGak} 
is closed under multiplication by $p$. Let $U$ be an open neighbourhood of $a$. By Lemma \ref{l:opntbk} we can suppose $U=a+n\Z$ with $n\in G_{a,\kappa}$; besides, Lemma \ref{l:nGab} shows that 
$np\in G_{b,\kappa}$ holds for every $b$ in $a+pn\Z$. Therefore, by Lemma \ref{l:opntbk},  all cosets appearing in the equality
$$a+pn\Z=\coprod_{i=0}^{p-1}\big(a+ipn+p^2n\Z\big)$$
are open, proving that $\calt_{\calb,\kappa}$ is not locally connected at $a$.

For sufficiency, assume $\infty$ is not in the range of $\kappa$ and fix an open coset $U_0=a+m\Z$. Consider
$$n=\prod_{p|m}p^{\kappa(p)}\,.$$
Then $U=a+n\Z$ is an open neighbourhood of $a$ contained in $U_0$ and we claim it is superconnected. Indeed, if $V_1,\dots,V_r$ are open subsets of $U$, then, by definition of $\calt_{\calb,\kappa}$, 
for each $i$ we can find $m_i\in\M_\kappa$ such that  $V_i$ contains an open coset of $m_i\Z$. The inclusion $V_i\subseteq U$  yields $m_i=nk_i$ and, by construction of $n$, this product can be in 
$\M_\kappa$ only if $n$ and $k_i$ are coprime. Letting $k$ be the least common multiple of $k_1,\dots,k_r$, we obtain that each $V_i$ contains a coset $a_i+nk\Z$. Also, $k$ and $n$ are coprime: thus, 
by the Chinese remainder theorem, there is a ring isomorphism $\psi_{n,k}$ as in \eqref{e:CRT}. The family $\{\calt_n\}$ attached to $\calt_{\calb,\kappa}$ is CRT-compatible (Remark \ref{r:bncm}.4) 
and $\pi_{nk}(U)$ is a fiber of the projection $\pi_n^{nk}$: it follows that $(\pi_{nk}(U),\calt_{nk})$ and $(\Z/k\Z,\calt_k)$ are homeomorphic. In particular, since $B_k\neq\emptyset$, both spaces 
have the property that nonempty closed sets have nonempty intersection. One concludes by lifting to $\Z$ the closures of the sets $\{[a_i]_{nk}\}\subseteq\pi_{nk}(V_i)$ and applying Lemma 
\ref{l:brwntyp}. \end{prf}

\begin{rmks} \label{r:aadmszc} The statement of Theorem \ref{t:glmkrc} is actually weaker than what its proof yields. We note the following facts.\\
\noindent{\bf 1.} The hypotheses $\kappa(p)=\infty$ and $[a]_{p^e}\notin B_{p^e}$ imply that $a+p^en\Z$ is a totally separated open set for every $n\in G_{a,\kappa}$ (because it can be written as a 
disjoint union of open cosets of $p^rn\Z$, for every $r\ge e$). In the setting of $(\N_1,\calt_G)$, one recovers \cite[Teorema 4.19]{aadm} as a special case.\\
\noindent{\bf 2.} If $v_p(n)\ge\kappa(p)$ holds for every $p\in\supp(n)$, then $A=a+n\Z$ is superconnected for any choice of $a\in\Z$. This is because if $A\cap V\neq\emptyset$, with $V$ an open 
coset, then the intersection must be of the form $b+nk\Z$, with $n$ and $k$ coprime: hence one can reason as in the proof of Theorem \ref{t:glmkrc}. In particular, if $\kappa$ is identically $1$ then 
all arithmetic progressions are connected in $(\N,\calt_{\calb,\kappa})$, as noted in \cite[Theorem 3.5]{szc10} in the case of $\calt_K$. \end{rmks}

Putting together Proposition \ref{p:cnncsq} and Remark \ref{r:aadmszc} we obtain the following dichotomy.

\begin{prop} \label{p:autaut} Let $\calb$ be a Golomb system. Then, for any $a\in\Z$ and $n\in\N_1$, the space $(a+n\Z,\calt_\calb)$ is either superconnected or totally separated. \end{prop}

\begin{prf} By Proposition \ref{p:cnncsq}, we know that $[a]_n\in B_n$ implies that $a+n\Z$ is superconnected. On the other hand, if $[a]_n\notin B_n$ then, by \eqref{e:glmsyst3} (and Remark 
\ref{r:bncm}.5) there must be a prime $p$ in $\supp(n)$ such that $[a]_{p^{v_p(n)}}$ lies outside $B_{p^{v_p(n)}}$. Hence $b+p^{v_p(n)+r}\Z$ is open for any $b$ in $a+n\Z$ and $r\in\N$, by Lemmata 
\ref{l:nGab} and \ref{l:opntbk}, and we have
$$(a+n\Z)\cap(b+p^{v_p(n)+r}\Z)=b+p^rn\Z\,,$$
showing that one can find more and more refined covers of $a+n\Z$ by disjoint open sets. \end{prf}

For topologists, the main interest in $(\N_1,\calt_G)$ and $(\N_1,\calt_K)$ is probably that they are instances of countable connected Hausdorff spaces. It becomes natural to ask when this happens 
for $(\Z,\calt_{\calb,\kappa})$ and its subspaces. 

\begin{prop} \label{p:glmHd} The space $(\Z,\calt_{\calb,\kappa})$ is Hausdorff if and only if for any $a,b\in\Z$ one can find a prime $p$ and an integer $r$ in the interval $[\gamma(p),\kappa(p)]$ 
such that $[a]_{p^r}$ and $[b]_{p^r}$ are two distinct points, both lying outside $B_{p^r}$\,. \end{prop}

\begin{prf} It is obvious that the condition is sufficient: if it holds, then the cosets $a+p^r\Z$ and $b+p^r\Z$ are both open in $\calt_{\calb,\kappa}$ and they are disjoint.

Vice versa, if $\calt_{\calb,\kappa}$ is Hausdorff, let $U,V$ be disjoint open neighbourhoods of $a,b\in\Z$. By Lemma \ref{l:opntbk}, we can assume $U=a+n\Z$ and $V=b+m\Z$, with $n$ and $m$ in 
$G_{a,\kappa}$ and $G_{b,\kappa}$ respectively. Let $d$ be the greatest common divisor of $m$ and $n$. An elementary reasoning shows that $U\cap V\neq\emptyset$ is equivalent to $a\equiv b$ mod $d$. 
So we must have $[a]_d\neq[b]_d$, which, by the Chinese remainder theorem, is possible only if there is a prime $p\in\supp(d)$ such that $a$ and $b$ are distinct modulo $p^{v_p(d)}$ and hence modulo 
$p^r$ with $r=\max\{\gamma(p),v_p(d)\}$. As a common divisor of $m$ and $n$, the prime $p$ is in $S_a\cap S_b$ and \eqref{e:dfSa}, \eqref{e:glmsyst4} and \eqref{e:glmsyst2} show that $B_{p^r}$ does 
not contain the images of either $a$ or $b$. Finally, $r\le\kappa(p)$ holds because the latter quantity is an upper bound for both $v_p(m)$ and $v_p(n)$. \end{prf}

\begin{eg} \label{e:intrsAB} The Golomb systems of nilpotents and of units of Example \ref{eg:brwntyp} do not not give rise to Hausdorff topologies on $\Z$, since either $[0]_n$ or $[1]_n$ are in 
$B_n$ for every $n$. However, it is easy to find instances of $\calb$ such that $\calt_{\calb,\kappa}$ is Hausdorff on all of $\Z$ for every $\kappa$. E.g., for each prime $p$ write $\Z/p\Z=X_p\sqcup 
Y_p$, where $X_p$ consists of those elements with a representative in the interval $[1-\sqrt{p},\sqrt{p}-1]$, so that neither $X_p$ nor $Y_p$ is ever empty. Let $\calp=\calp_1\sqcup\calp_2$ be any 
partition of the primes into two infinite 
subsets and consider the Golomb system defined by $\gamma(p)$ identically $1$ and
$$B_p:=\begin{cases} X_p & \text{ if }p\in\calp_1\,; \\ Y_p & \text{ if }p\in\calp_2\,. \end{cases}$$
Also, let $B_p'$ be the complement of $B_p$ in $\F_p$. Then, whatever is $\kappa$, both $\calt_{\calb,\kappa}$ and $\calt_{\calb',\kappa}$ are connected Hausdorff topologies on $\Z$. \end{eg}

\subsubsection{The Szczuka dual} \label{sss:szd} It was noted in \cite[\S5]{szc13} that an arithmetic progression is connected with respect to $\calt_G$ if and only if it is in the base used 
to define $\calt_S$ and vice versa. This relationship can be axiomatized in the following way: given a Golomb system $\calb=\{B_n\}$, define $\calb'=\{B_n'\}$ by \eqref{e:glmsyst2}, \eqref{e:glmsyst3} 
and \eqref{e:glmsyst4}, starting with $B_{p^{\gamma(p)}}'$ the complement of $B_{p^{\gamma(p)}}$ if $\gamma(p)>0$, so to have
\begin{equation} \label{e:Bdlpr} \Z/p^r\Z=B_{p^r}\sqcup B_{p^r}'\;\;\text{ if }r\ge\gamma(p)\ge1\,. \end{equation}
(If $\gamma(p)=0$, then $B_{p^r}$ and $B_{p^r}'$ are both equal to $\Z/p^r\Z$, for every $r\in\N$.) We say that $\calb'$ is the dual of $\calb$ and we call $\calt_{\calb'}$ the Szczuka dual of 
$\calt_\calb$. The primal example arises with $B_p=\{[0]_p\}$ and $B_p'=\F_p^*$, yielding $\calt_\calb=\calt_G$ and $\calt_{\calb'}=\calt_S$\,. Another instance, with $\calt_\calb$ and 
$\calt_{\calb'}$ both Hausdorff, was introduced in Example \ref{e:intrsAB}.

\begin{prop} \label{p:szczdlt} Let $\calb$ be a Golomb system and $\calb'$ its dual. Assume $\gamma(p)=1$ for every $p$. Then a coset $a+n\Z$ is superconnected with respect to $\calt_\calb$ 
if and only if it is open in $\calt_{\calb'}$\,; and vice versa. \end{prop}

\begin{prf} Just apply Lemma \ref{l:opntbk} and Proposition \ref{p:cnncsq} together with the chain of equivalences
$$ n\in G_a(\calb') \stackrel{(*)}{\Longleftrightarrow} [a]_{p^{v_p(n)}}\notin B_{p^{v_p(n)}}'\,\text{if }p|n \stackrel{(\dagger)}{\Longleftrightarrow} 
[a]_{p^{v_p(n)}}\in B_{p^{v_p(n)}}\,\text{if }p|n \stackrel{(\ddagger)}{\Longleftrightarrow} [a]_n\in B_n$$
where $(*)$ is \eqref{e:nGass}, $(\dagger)$ holds by \eqref{e:Bdlpr} and $(\ddagger)$ is just a rephrasing of \eqref{e:glmsyst3}. \end{prf}

The result is also true (with the same proof) restricting the topology to $\N_1$ and replacing $b+k\Z$ with its positive elements: hence we recover Szczuka's observation in \cite[\S5]{szc13} as a 
particular case. 

\begin{rmks} The hypothesis on $\gamma$ is needed for $(\dagger)$ to hold: if either $\gamma(p)=0$ or $\gamma(p)>r\ge1$ then \eqref{e:Bdlpr} fails because the sets $B_{p^r}$ and $B_{p^r}'$ are not 
disjoint (by \eqref{e:glmsyst4} and the assumption of Remark \ref{r:gmmB}). Renouncing this hypothesis, we have the following results.\\
\noindent{\bf 1.} The equivalence $(\dagger)$ is still true if $n\in\M$ and $\gamma(p)>0$ for every $p$ in $\supp(n)$. Hence for such an $n$ we have that a coset of $n\Z$ is superconnected in one 
topology if and only if it is open in the dual one.\\
\noindent{\bf 2.} The implication $n\in G_a(\calb')\Rightarrow[a]_n\in B_n$ holds for any $n\in\N_1$ and any Golomb system $\calb$, because the ``$\Rightarrow$'' part of $(\dagger)$ only uses the 
fact that $B_{p^r}$ and $B_{p^r}'$ cover $\Z/p^r\Z$. Thus if a coset $a+n\Z$ is open in a topology then it is superconnected in its dual. \end{rmks}

\subsection{Arithmetic applications} \label{ss:applct} We briefly mention a few results of arithmetic interest regarding some of the topologies we have surveyed. Note that, in the opposite direction, 
papers like \cite{bmt}, \cite{bst} and \cite{st} use highly non-trivial facts from number theory to deduce topological applications, namely to study self-homemomorphisms of $\N_1$ or $\Z$ with 
respect to $\calt_G$ and $\calt_K$. 

\subsubsection{Furstenberg} \label{sss:appltpF} The most famous application of $\calt_F$ is to show that $\calp$ is infinite. The idea is to consider the set of non-invertible integers
\begin{equation} \label{e:znuprm} \Z-\{\pm1\}=\bigcup_{p\in\calp}p\Z  \end{equation}
and note that it would be closed if $\calp$ were finite (since ideals are closed for $\calt_F$), making $\{\pm1\}$ open, in contradiction with the fact that open sets cannot be finite. In 
\cite{clrk}, this argument and its relation with Euclid's classical proof are scrutinized from the viewpoint of commutative algebra.

In \cite[\S4 and \S5]{brou} one can find a number of results on the closures and cluster points of some sets of interest to number theorists (such as $\calp$, integers with $k$ prime factors, $k$th 
powers, Fermat and Mersenne numbers, squarefree integers): we will recover and extend some of these statements later, in Theorem \ref{t:23} and Proposition \ref{p:98}. Broughan's findings about 
primes are furthered in \cite[\S5]{szc10}, where it is shown that $(\calp,\calt_F)$ has empty interior (as no arithmetic progression can consist only of primes) and that it is locally connected (for 
the trivial reason that it is discrete, since each singleton $\{p\}=p\Z\cap\calp$ is open).

\begin{rmk} Many authors have provided convincing arguments that Furstenberg's proof is substantially equivalent to the one by Euclid and that his use of topology adds really nothing to the 
reasoning: see for example the discussion in \cite[page 204]{clrk} and the papers cited there. As observed also in \cite{carls}, both proofs in the end reduce to \eqref{e:znuprm}. We cannot 
disagree. However, we contend that the topological framework does indeed provide valuable new insight, leading to new questions with interesting answers, as we will try to show in subsection 
\ref{ss:clsP} below. \end{rmk}

\subsubsection{Golomb} \label{sss:glmar} In Golomb's topology open sets must be infinite and ideals are closed: hence Furstenberg's proof that $\calp$ is infinite works verbatim replacing $\calt_F$ 
with $\calt_G$\,. However, $\calt_G$ might appear more interesting than $\calt_F$ for number theory because of its link with Dirichlet's theorem on primes in arithmetic progressions, which, by 
\cite[Theorem 6]{gol1}, is equivalent to the statement that $\calp$ is dense in $(\N_1,\calt_G)$. Indeed, Dirichlet's theorem can be easily reduced to the claim that every coset $a+b\Z$, with $a,b$ 
coprime, contains one element of $\calp$ and now it is enough to note that, by definition, such cosets form a base of $\calt_G$\,. For extensions of this argument to other integral domains, see 
\cite[\S2 and \S4]{kp} and \cite{mrkprb}. In \cite{orum} it is shown, by a clever use of the continuity of multiplication, that the density of $\calp$ in $\N_1$ could be obtained if one knew that any 
product of two primes is a cluster point of $\calp$.

The interior of $\calp$ is empty in $(\N_1,\calt_G)$: this result, which can be easily deduced from the analogous fact for the finer topology $\calt_F$, was originally proved in \cite[Theorem 
7]{gol1} and has been extended to other domains in \cite[Theorem 20]{kp}. Golomb also showed that the set of positive integers $m$ such that $6m-1$ and $6m+1$ are a pair of twin primes is closed in 
$(\N_1,\calt_G)$ - see \cite[Theorem 8]{gol2}. Knopfmacher and Porubsk\'y noted that $\calt_G$ can be applied to prove that certain sets of (products of) pseudoprimes are infinite: we refer to 
\cite[pages 146-147]{kp} for details.

The space $(\calp,\calt_G)$ is proved to be metrizable and homeomorphic to $\Q$ (with the topology induced by the canonical embedding in $\R$) in \cite[Theorem 3.1]{bmt} (weaker results were earlier 
given in \cite[\S5]{szc10}). More interesting (because of its relation with a failure of the local-global principle) is the fact that sets of $k$th powers, which are always closed in $\calt_F$ by 
\cite[Theorem 4.1]{brou}, need not be closed with respect to $\calt_G$ - see \cite[Proposition 4.8]{bmt}.

A further application of $\calt_G$ to the study of subsets of $\calp$, proposed in \cite[Section III]{gol2}, will be discussed in \S\ref{sss:gol2S3}.

\subsubsection{Kirch} Furstenberg's argument can be repeated also in the case of $\calt_K$, since the ideals $p\Z$, with $p$ prime, are closed also in this topology. Moreover, as noted in \cite[page 
96]{szsz}, $\calp$ is dense with empty interior in $(\Z,\calt_K)$, because it is so in the stronger topology $\calt_G$. In \cite[\S5]{szc10} it is proved that $(\calp,\calt_K)$ is neither connected 
nor locally connected.

\subsubsection{Rizza} As observed in \cite[Corollary 7.8]{szc16}, the topology $\calt_R$ makes $\calp$ disconnected and not compact. Actually, by \eqref{e:clsrR}, for every nonempty 
$\calp_1\subseteq\calp$ one has 
$$\overline{\calp_1}^{\calt_R}=\calp_1\sqcup\{\pm1\}\sqcup\{-p:p\in\calp_1\}\,.$$
It follows immediately that $(\calp,\calt_R)$ is a discrete topological space.

\begin{prop} \label{p:rzzcrtr} The following are equivalent: \begin{itemize}
\item $\calp$ is infinite;
\item $(\calp,\calt_R)$ is not compact;
\item $(\Z-\{\pm1\},\calt_R)$ is not compact. \end{itemize} \end{prop}

\noindent This can be seen as a special case of the compactness criterion in \cite[Theorem 7.4]{szc16}. For the reader's convenience, we provide an independent proof.

\begin{prf} Since $(\calp,\calt_R)$ is discrete, we only need to prove that $\calp$ is finite if and only if $(\Z-\{\pm1\},\calt_R)$ is compact. In one direction, since ideals are open in $\calt_R$, 
if $\calp$ is infinite then \eqref{e:znuprm} provides an open cover with no finite subcover. As for the converse, note that every nonempty closed subset of $(p\Z,\calt_R)$ must contain $p$. Hence 
$p\Z$ is compact and therefore if $\calp$ is finite \eqref{e:znuprm} yields that $\Z-\{\pm1\}$ is compact too. \end{prf}

We don't know if Proposition \ref{p:rzzcrtr} can lead to a new proof of the existence of infinitely many primes.\\

We also mention that the set of squarefree integers is closed in $\calt_R$ (its complement is the union of the ideals generated by $n^2$, for $n>1$).

Probably more interesting is the fact that a function $f\colon\Z\rightarrow\Z$ is continuous with respect to the topology of \cite{rzz} if and only if it it satisfies
\begin{equation} \label{e:dvsbrzz} a|b\,\Longrightarrow\;f(a)|f(b) \end{equation}
(\cite[Proposition 5]{rzz}; see also the generalizations in \cite[Theorem 13]{prb00} and \cite[Theorem 2.1]{hauk}). So, for example, completely multiplicative functions, Euler's totient 
$\varphi\colon\N\rightarrow\N$ and the M\"obius function $\mu\colon\Z\rightarrow\Z$ are continuous (\cite[Propositions 6, 7 and 9]{rzz}; see also \cite[\S3]{hauk} and \cite[pages 398-399]{prb00}). 
This suggests the use of $\calt_R$ and its preorder-induced analogues (as in Remark \ref{r:rzz}) in relation with convolution products (\cite[page 185]{rzz}; see also \cite{prb00}). In 
\cite[\S4,5]{hauk} readers can find some applications to GCD-matrices.

\begin{rmk} If $f\colon\Z\rightarrow\Z$ is continuous with respect to $\calt_R$ then \eqref{e:dvsbrzz} holds, but the converse is false (for a simple example, take $f(n)=1$ for $n\neq0$ and 
$f(0)=0$). This is because in defining $\calt_R$ we did not take the ideal $\{0\}$ to be open, while \cite{rzz} does: $\calt_R$-continuity becomes equivalent to \eqref{e:dvsbrzz} only after 
restricting the domain to $\Z-\{0\}$. However, the functions $\varphi$ and $\mu$ are still continuous under $\calt_R$, as interested readers can easily check. \end{rmk}

\subsubsection{Szczuka} In \cite[page 97]{szsz} it is proved that the closure of $\calp$ in $(\N_1,\calt_S)$ is $\calp\cup\{1\}$. We already noted that $\{\pm1\}$ is contained in any closed subset 
of $(\Z,\calt_S)$; observing that $4\Z$ is a neighbourhood of $0$ and that, for $a\neq0$, the coset $a+a^4\Z$ is open, with $(a+a^4\Z)\cap\calp$ either empty or $\{p\}$ (if $a=p\in\calp$) when 
$a\neq\pm1$, we obtain $\overline\calp^{\calt_S}=\calp\cup\{\pm1\}$. This also proves that $(\calp,\calt_S)$ is discrete and hence both disconnected and locally connected (\cite[Theorems 5.1 and 
5.2]{szsz}).

\section{The algebraic inverse limit and compactifications}

\subsection{The ring $\za$} \label{ss:za} This subsection contains nothing new: actually, most of the material can be found in standard textbooks. We include it to fix some notation and 
make the paper accessible to a wider audience.

\subsubsection{Inverse limits and $\za$}  We briefly recall the notion of inverse limit.

A directed set is a preorder such that for every $i,j$ there is $k$ with $i\le k$ and $j\le k$. Let $A_i$ be objects in a category $\cala$, indexed by a directed set $I$, and let  $f_i^j\colon A_j\to 
A_i$ be arrows in $\cala$ for all $ i\le j$ such that $f_i^i$ is the identity on $A_i$ and $f_i^k=f_i^j\circ f_j^k$ for all $i\le j\le k$. Then $\big((A_i)_i,(f_i^j)_{i\le j}\big)$ is called an 
inverse system over $I$. 

Given an inverse system as above, assume there is a pair $\big(\varprojlim A_i,(\hat f_i)_i\big)$, where $\varprojlim A_i$ is an object in $\cala$ and the $\hat f_i$'s are morphisms $\hat 
f_i\colon\varprojlim A_i\rightarrow A_i$ such that $\hat f_i=f_i^j\circ\hat f_j$ for all $i\le j$, satisfying the following universal property: for any other such pair $(X,\varpi_i)$ there exists a 
unique morphism $\varpi\colon X \rightarrow\varprojlim A_i$ such that the diagram
\begin{equation} \label{e:78} \xymatrixcolsep{3pc}\xymatrixrowsep{3pc}
\xymatrix{ & X \ar[d]^{\exists!\varpi} \ar@/_1pc/[ddl]_{\varpi_k} \ar@/_2pc/[dd]_{\varpi_j} \ar@/^1pc/[ddr]^{\varpi_i}\\
&\varprojlim A_i \ar[ld]_{\hat f_k} \ar[d]^{\hat f_j} \ar[rd]^{\hat f_i} \\
A_k \ar[r]_{f_j^k} & A_j \ar[r]_{f_i^j} & A_i } \end{equation}
commutes. Then we say that $\big(\varprojlim A_i,(\hat f_i)_i\big)$ is the inverse limit of the system $\big((A_i),(f_i^j)\big)$.

\begin{rmk} \label{r:IJinvlim} Let $I$ be a directed set as above and $J\subseteq I$ a directed subset: then one can also consider the inverse limit $\varprojlim_{j\in J}A_j$\,. By the universal 
property, the arrows $(\hat f_j)_{j\in J}$ induce a unique morphism
\begin{equation} \label{e:IJinvlim} \varprojlim_{i\in I}A_i\longrightarrow\varprojlim_{j\in J}A_j\,. \end{equation}
If, moreover, $J$ is cofinal (that is, for every $i\in I$ there is $j\in J$ such that $i\le j$), then the universal property can be applied, once again, to show that the arrow in \eqref{e:IJinvlim} 
has an inverse and therefore is an isomorphism. \end{rmk}

The set $\N_1$ is directed with the order induced by divisibility, $n\le m\Leftrightarrow n|m$. It is clear that $\big((\Z/n\Z),(\pi_n^{nk})\big)$ form an inverse system in the category of rings.

\begin{thm} \label{t:exza} The inverse system $\big((\Z/n\Z),(\pi_n^{nk})\big)$ has an inverse limit $\big(\varprojlim\Z/n\Z,(\hpi_n)_n\big)$. \end{thm}

\begin{skprf} Consider the set
\begin{equation} \label{e:zasbst} \za:=\left\{(a_n)\in\prod_{n\in\N_1}\Z/n\Z : \pi_n^{nk}(a_{nk})=a_n\;\forall\,n,k\in\N_1\right\}\,, \end{equation}
together with the maps
$$\hpi_n\colon\za\rightarrow\Z/n\Z$$
induced by the projections $\prod_n\Z/n\Z\twoheadrightarrow\Z/n\Z$. It is easy to see that $\za$ is a subring of $\prod_n\Z/n\Z$ and that $\big(\za,(\hpi_n)_n\big)$ satisfies the universal property 
described by \eqref{e:78}. \end{skprf} 

We would like to emphasize that it is psychologically more convenient to think of $\za$ as defined by the universal property rather than by \eqref{e:zasbst} (just as one thinks of $\R$ as defined by 
the real numbers axioms rather than obtained from $\Q$ via Dedekind cuts).

\begin{rmk} \label{r:pdc} With trivial changes, the proof of Theorem \ref{t:exza} shows that inverse limits exist in the category of rings. In particular, we can reason like in Remark 
\ref{r:IJinvlim} and obtain an inverse limit from the inverse system attached to any submonoid $G$ of $\N_1$. When $G$ is generated by a prime number $p$, the resulting limit is the ring of $p$-adic 
integers 
$$\Z_p:=\varprojlim_{n\in\N}\Z/p^n\Z\,.$$ 
We will write
$$\hpi_{p^\infty}\colon\za\longrightarrow\Z_p$$
to denote the map in this instance of \eqref{e:IJinvlim}. (The notation comes from the fact that $\hpi_{p^\infty}$ is the limit of the maps $\hpi_{p^n}$ as $n$ grows.) \end{rmk}

\begin{lem} \label{l:iotza} The map $\iota$ of \eqref{e:dfniota} has image inside $\za$. \end{lem}

\begin{prf} Obvious from \eqref{e:zasbst}. \end{prf} 

\begin{cor} All the maps $\hpi_n\colon\za\rightarrow\Z/n\Z$ are surjective. \end{cor}

\begin{prf} Obvious from $\pi_n=\hpi_n\circ\iota$. \end{prf}

In the following, we will often tacitly identify $\iota(\Z)$ with $\Z$, so to consider the latter as a subring of $\za$.

\subsubsection{Arithmetic of $\za$} \label{sss:arthza} It is well-known that, for every prime $p$, the ring $\Z_p$ has a number of nice properties: it is a principal ideal domain, containing $\Z$ 
(via the map $\iota_p:=\hpi_{p^\infty}\circ\iota$) and where every ideal is generated by some power of the prime $p$. Moreover, the $p$-adic valuation on $\Z$ extends to 
$v_p\colon\Z_p\rightarrow\N\cup\{\infty\}$, with the property
$$v_p(x)=0\iff x\in\Z_p^*\,.$$
Putting $\hvp:=v_p\circ\hpi_{p^\infty}$ we obtain a $p$-adic valuation $\hvp\colon\za\rightarrow\N\cup\{\infty\}$.

The correlation of $\za$ with the various $\Z_p$ is clarified by the following result.

\begin{thm} \label{t:440} There is a canonical ring isomorphism
\begin{equation} \label{e:100} \za\simeq \prod_{p\in\calp}\Z_p\,, \end{equation}
given by $x\mapsto\big(\hpi_{p^\infty}(x)\big)_{p\in\calp}$\,.\end{thm}

\begin{proof} By the fundamental theorem of arithmetic, we know that for all $n\in\N_1$ there is a unique factorization in prime numbers $n=\prod_pp^{v_p(n)}$. Together with the Chinese remainder 
theorem this yields a canonical ring isomorphism
$$\Z/n\Z\simeq \prod_{p\in\calp}\Z/p^{v_p(n)}\Z\,.$$
If $n$ divides $m$ we obtain commutative diagrams
\begin{equation} \label{e:120} \begin{CD} \Z/m\Z @>\sim>> \prod_p \Z/p^{v_p(m)}\Z \\
@V{\pi_n^m}VV @VV{\prod\pi_{p^{v_p(n)}}^{p^{v_p(m)}}}V\\
\Z/n\Z @>\sim>>  \prod_p\Z/p^{v_p(n)}\Z\;.  \end{CD} \end{equation}

Let $R$ be any ring and $f_n\colon R\rightarrow\Z/n\Z$ a system of homomorphisms such that $f_n=\pi_n^{nk}\circ f_{nk}$ for all $n,k\in\N_1$\,\!. Then for each $p$, by the universal property of 
$\Z_p$, there is a unique homomorphism $\hat f_p\colon R\rightarrow\Z_p$ factoring $f_{p^i}$ for all $i$. By \eqref{e:120} it follows that every $f_n$ factors through $\prod_p\hat f_p$\,, showing 
that $\prod_p\Z_p$ enjoys the same universal property as $\za$.

The isomorphism \eqref{e:100} is obtained by taking $R=\za$ and $f_n=\hpi_n$. The  explicit expression for it can be checked by a little diagram-chasing in \eqref{e:120}. \end{proof}

We also note that the decompositions in \eqref{e:100} and \eqref{e:120} immediately yield $\ker(\hpi_n)=n\za$. Besides, \eqref{e:100} implies that maximal ideals of $\za$ have the form $p\za$, with 
$p\in\calp$.

\begin{rmk} Let $G$ be a submonoid of $\N_1$ and $\calp_G\subseteq\calp$ the set of primes dividing elements of $G$. Then the same reasoning used to prove Theorem \ref{t:440} shows
\begin{equation} \label{e:Gbrgh} \varprojlim_{n\in G}\Z/n\Z\simeq\prod_{p\in\calp_G}\Z_p\,. \end{equation}
In the case when $G$ equals the submonoid $G'$ generated by $\calp_G$, this statement can be found as part of \cite[Theorem 3.2]{brou}. Furthermore, \cite[Theorem 3.3]{brou} can be strengthened 
(again, with no hypotheses on $G$) by observing that the image of $\Z$ in the inverse limit on the left-hand side of \eqref{e:Gbrgh} corresponds, on the right-hand side, to the diagonal embedding in 
$\prod\Z_p$\,. \end{rmk}

Recall that a ring is {\em B\'ezout} if every finitely generated ideal is principal.

\begin{cor} \label{c:zabzt} The ring $\za$ is B\'ezout. \end{cor}

\begin{proof} Consider the finitely generated ideal $(x_1,...,x_n)\subseteq\za$. For every prime $p$, define
$$m_p:=\min\{\hvp(x_1),...,\hvp(x_n)\}\,.$$
By Theorem \ref{t:440}, there is a unique $x\in\za$ such that $\hpi_{p^\infty}(x)=p^{m_p}$ for every $p$ (with the convention $p^\infty=0$).
The equality $(x_1,..,x_n)=x\za$ is easily checked by means of \eqref{e:100}. \end{proof}

\begin{rmk} Since $\calp$ is infinite, it follows immediately from \eqref{e:100} that $\za$ is not Noetherian: for an example of an ideal which is not finitely generated, consider the set of all $x$ 
such that $\hpi_{p^\infty}(x)=0$ for almost every $p$. Thus Corollary \ref{c:zabzt} can be seen as a weak form of unique factorization in a non-Noetherian context.

We also observe that Corollary \ref{c:zabzt} applies as well to the ring described in \eqref{e:Gbrgh} (which is Noetherian if and only if $\calp_G$ is finite). \end{rmk}

For future reference, we note one further immediate consequence of Theorem \ref{t:440}: for any $a,b\in\za$ one has
\begin{equation} \label{e:ltrelr} a+b\za\simeq\prod_p(a+b\Z_p)=\prod_{p|b}(a+p^{v_p(b)}\Z_p)\times\prod_{p\nmid b}\Z_p\,. \end{equation}
In order to lighten notation, in the following we will often identify the two sides of \eqref{e:100}, thinking of $\za$ as the product of the $\Z_p$'s: since  the isomorphism is canonical, this 
causes no ambiguity.

\subsection{The profinite topology} \label{ss:zatplg} The ring $\za$ has a natural topology as a subset of $\prod_n\Z/n\Z$ with the product topology $\prod_n\calt_{F,n}$ (the Furstenberg topology). 

Since the transition morphisms are continuous, $\big((\Z/n\Z,\calt_{F,n}),(\pi_n^{nk})\big)$ is an inverse system also in the category of topological spaces; one checks immediately that the maps 
$\hpi_n$ are continuous as well and that the universal property of $\big(\za,(\hpi_n)\big)$ holds also with respect to this topological structure. Moreover the ring operations on 
$(\prod_n\Z/n\Z,\prod_n\calt_{F,n})$, and hence on $\za$, are continuous. Therefore $\za$ is a topological ring and it is the inverse limit of $\big((\Z/n\Z,\calt_{F,n}),(\pi_n^{nk})\big)$ 
both as a ring and as a topological space. 

It might be convenient to note that, as $n$ varies in $\N_1$ and $a$ in $\Z$, the sets
\begin{equation} \label{e:bsza} a+n\za=\hpi_n^{-1}([a]_n) \end{equation}
provide a base for the topology on $\za$. This also shows that $\iota(\Z)$ is dense in $\za$. As a dense subset of $\Z$, it follows that $\N$ is dense in $\za$ as well. 

As a limit of finite discrete topological spaces, the topology on $\za$ is usually called {\em profinite topology}.

\begin{rmk} As a fun fact, we mention an observation of Lenstra (see \cite[\S8]{lenstr}): \begin{itemize}
\item $\prod_n\Z/n\Z\simeq\big(\prod_n\Z/n\Z\big)/\za$ as additive topological groups;\vspace{1pt}
\item $\prod_n\Z/n\Z\simeq\za\times\prod_n\Z/n\Z$ as groups but not as topological groups. \end{itemize}
In particular, this shows that $\za$ is ``much smaller'' than $\prod_n\Z/n\Z$.  \end{rmk} 

We also note that, giving to each $\Z_p$ the usual $p$-adic topology, the maps $\hpi_{p^\infty}$ are all continuous and \eqref{e:100} is an isomorphism of topological rings (with the product 
topology on $\prod_p\Z_p)$.

\begin{prop} The ring $\za$ is a closed subset of $\big(\prod_n\Z/n\Z,\prod_n\calt_{F,n}\big)$. \end{prop}

\begin{prf} For each $m$, let $\tilde\pi_m\colon\prod_n\Z/n\Z\twoheadrightarrow\Z/m\Z$ denote the natural projection. By definition of product topology, $\tilde\pi_m$ is continuous. Moreover, in 
the Furstenberg topology, every quotient $\Z/n\Z$ is Hausdorff and the maps $\pi_n^{nk}$ are all continuous. It follows that each equality
$$\pi_n^{nk}\circ\tilde\pi_{nk}=\tilde\pi_n$$
defines a closed subset $C_{n,k}$ in $\prod_n\Z/n\Z$. By \eqref{e:zasbst} we have $\za=\bigcap_{n,k}C_{n,k}$\,.  \end{prf}

\begin{cor} As a topological space, $\za$ is compact. \end{cor}

\begin{prf} By Tychonoff's theorem, $(\prod_n\Z/n\Z,\prod_n\calt_{F,n})$ is compact. As a closed subset, it follows that $\za$ is compact. \end{prf}

\begin{rmk} As a topological space, $\za$ is also metrizable: actually, one can observe that the space $(\prod_n\Z/n\Z,\prod_n\calt_{F,n})$ is itself metrizable, since it is compact Hausdorff and 
sets of the form
$$\prod_{n\in I}\{[a_n]_n\}\times\prod_{n\notin I}\Z/n\Z$$
(with $I$ varying among finite subsets of $\N_1$ and $a_n\in\Z$) provide a countable base. We will not define a metric on $\za$, since it is easier to work with topology alone. 
\footnote{In fancy words, ``the native hue of topology is sicklied o'er with the pale cast of a metric''. More seriously, a metric might even obscure some important properties: for example, 
the one on $(\Z,\calt_F)$ introduced  in \cite[page 111]{lm} leads to a description of the completion (which is homeomorphic to $\za$) as $\prod_{n\ge2}\Z/n\Z$, in \cite[Theorem 7]{lm}. (The proof 
is equivalent to showing that every element in $\za$ has a unique representation as a convergent series $\sum_{n\ge1}a_nn!$ with $a_n\in\{0,\dots,n\}$.) While formally correct, such a description 
has the disadvantage of making less transparent the algebraic structure of $\za$. For a familiar analogue, just think of using the decimal representation to describe the real line $\R$ as a quotient 
of the product of $\Z$ and infinitely many copies of $\Z/10\Z$.} Interested readers can obtain a metric by using the isomorphism \eqref{e:100} to combine all the $p$-adic absolute values 
as a family of seminorms; the construction of \cite[Theorem 3.1]{brou} provides an alternative approach. 

Since it contains $\Z$ as a dense subset, $\za$ can be built as its completion.

As topological spaces, $(\Z,\calt_F)$ and $\za$ are homeomorphic respectively to $\Q$ (with the usual topology induced by the Euclidean metric) and to the Cantor set: see \cite[Theorem 2.4 and 
Corollary 3.1]{brou} for proofs. (Note that these results actually hold for all limits over a monoid $G$ as in \eqref{e:Gbrgh}, and for the topologies induced by them on $\Z$; in \cite{brou}, such 
topologies are called ``flat''.) \end{rmk}

\subsubsection{Compactification of p\'eij\'i topologies} \label{sss:cmptpj} Let $\calt$ be a  p\'eij\'i topology and $\{\calt_n\}$ the maximal family attached to it. In general $(\Z,\calt)$ is not a 
compact topological space. However, each $(\Z/n\Z,\calt_n)$ is compact and Tychonoff's theorem implies that so is $(\prod\Z/n\Z,\prod\calt_n)$. It follows that the closure of $\iota(\Z)$ in 
$(\prod\Z/n\Z,\prod\calt_n)$ is a compactification of $\Z$, which we will denote by $\za^\calt$.

\begin{lem} \label{l:pjclsr} Let $\calt$, $\{\calt_n\}$ be as above. Given an element $x=(x_n)\in\prod\Z/n\Z$, for each $n\in\N_1$ let $U_n$ be the smallest open neighbourhood of $x_n$ in 
$\calt_n$\,. Then
\begin{equation} \label{e:clsTns} x\in\za^\calt \iff \exists\,z\in\za\text{ such that }\hpi_n(z)\in U_n\;\forall\,n\,. \end{equation}
\end{lem}

\begin{prf} By definition of closure, $x\in\za^\calt$ is equivalent to say that every neighbourhood $V$ of $x$ contains an element of $\iota(\Z)$. In the product topology $\prod\calt_n$ we can 
assume, without loss of generality, that this neighbourhood has the form $V=\prod_nV_n$, where $V_n$ is a neighbourhood of $x_n$ and $V_n=\Z/n\Z$ for almost every $n$. For such a $V$, let 
$\{n_1,\dots,n_r\}$ be the set of those $n$'s such that $V_n\neq\Z/n\Z$ and let $m$ be the least common multiple of the $n_j$'s. If there is $z\in\za$ as in \eqref{e:clsTns}, take any $a\in\Z$ 
such that $[a]_m=\hpi_m(z)$. Then for $j=1,\dots,r$ we have
$$[a]_{n_j}=\pi^m_{n_j}([a]_m)=\pi^m_{n_j}(\hpi_m(z))=\hpi_{n_j}(z)$$
by the definition of $\za$ in \eqref{e:zasbst}. Hence $[a]_{n_j}$ is in $U_{n_j}\subseteq V_{n_j}$ and this proves $\iota(a)\in V$.

On the other hand, each set $U_n$ is closed in $\calt_{F,n}$ and hence so is $\hpi_n^{-1}(U_n)$ in $\za$. If $x\in\za^\calt$ then for any $n_1,\dots,n_r$ one can find an integer $a$ such that
$[a]_{n_i}\in U_{n_i}$ holds for all $i$, i.e., 
$$\iota(a)\in\bigcap_{i=1}^r\hpi_{n_i}^{-1}(U_{n_i})\,.$$ 
Since $\za$ is compact, the finite intersection property implies that there is $z$ in $\bigcap_n\hpi_n^{-1}(U_n)$.  \end{prf}

\begin{cor} \label{c:zazat} \begin{itemize} \item[]
\item[{\bf 1.}] The inclusion $\za\subseteq\za^\calt$ holds for every $\calt$.\vspace{1pt}
\item[{\bf 2.}] One has $\za=\za^\calt$ if and only if $\calt$ is the Furstenberg topology.
\end{itemize} \end{cor}

\begin{prf} The first statement is obvious from \eqref{e:clsTns}. As for the second, if $\calt\neq\calt_F$ then $\calt_k$ is not the discrete topology for at least one $k$. Thus for this $k$ there 
must be a point which is not closed: that is, there are $a,b\in\Z$ such that $[a]_k\neq[b]_k$ and $[b]_k\in\overline{\{[a]_k\}}$. Define $x=(x_n)\in\prod_z\Z/n\Z$ by 
$$x_n=\begin{cases} [a]_n & \text{ if }n\neq k\,;\\ [b]_k & \text{ if }n=k\,.  \end{cases}$$
Then the characterization \eqref{e:zasbst} shows $x\notin\za$, while \eqref{e:clsTns}, with $z=\iota(a)$, implies $x\in\za^\calt$.
\end{prf}

Corollary \ref{c:zazat} implies that $\Z$ is dense in $\za$ with respect to any p\'eij\'i topology. Also, it shows that one always has $\za^\calt\neq\Z$, even when $(\Z,\calt)$ is already compact. \\

For $X\subseteq\Z$, let $\xa^\calt$ be the closure of $X$ in $\za^\calt$; in the case $\calt=\calt_F$ we shall write simply $\xa$. As we are going to discuss briefly in subsection \ref{ss:clsP} (and, 
hopefully, more extensively in future work - see also \cite{DL}), determining $\xa$ appears quite meaningful for number theory. We don't know if the same is true for other p\'eij\'i topologies.

\subsubsection{Golomb systems, revisited} \label{sss:brsystr} Let $\{\calt_n\}_{n\in\N_1}$ be the maximal family originating a p\'eij\'i topology $\calt$ and define $\{B_n\}_{n\in\N_1}$ as in 
\S\ref{sss:brwntyp}. 

\begin{lem} \label{l:clsdB2} If condition {\bf B2} of Definition \ref{d:brwntyp} holds for a determining monoid $\M$, then there is a subset $B$ of $\za$, closed with respect to the profinite 
topology, such that $\hpi_n(B)=B_n$ for every $n\in\M$. \end{lem}

In particular, $B$ is nonempty.

\begin{prf} Take
$$B:=\bigcap_{n\in\M}\hpi_n^{-1}\big(B_n\big)\,.$$
Note that in the profinite topology $B$ is closed, because so is every subset of $\Z/n\Z$ with respect to $\calt_{F,n}$\,. Fix $n\in\M$ and $x\in B_n$. Then, for any $n_1,\dots,n_r\in\M$, the 
intersection
\begin{equation} \label{e:intchs} \hpi_n^{-1}(x)\cap\bigcap_{i=1}^r\hpi_{n_i}^{-1}\big(B_{n_i}\big) \end{equation}
is nonempty, because, by {\bf B2}, it contains $\hpi_m^{-1}\big(B_m\cap(\hpi_n^m)^{-1}(x)\big)$ (with $m$ any common multiple of $n_1,\dots,n_r$). Since in the profinite topology $\za$ is compact and 
all sets appearing in \eqref{e:intchs} are closed, the finite intersection property implies that $B\cap\hpi_n^{-1}(x)$ is nonempty. \end{prf}

\begin{cor} \label{c:bbn} Under the hypotheses of Lemma \ref{l:clsdB2}, the inclusion $\hpi_n(B)\subseteq B_n$ holds for every $n\in\N_1$\,. \end{cor}

\begin{prf} Fix $n\in\N_1$ and take $d$ as in Lemma \ref{l:ndm}. Also, let $m\in\M$ be a multiple of $d$. For $x\in B$, we have: \begin{itemize}
\item $\hpi_m(x)\in B_m$ by Lemma \ref{l:clsdB2};
\item $\hpi_d(x)=\pi_d^m(\hpi_m(x))\in B_d$ because $\pi_d^m(B_m)\subseteq B_d$ by Lemma \ref{l:bnkn};
\item $\hpi_n(x)\in(\pi_d^n)^{-1}(\hpi_d(x))\subseteq B_n$ since Lemma \ref{l:ndm} immediately implies $B_n=(\pi_d^n)^{-1}(B_d)$.
\end{itemize} \end{prf}

If moreover $\calb=\{B_n\}_{n\in\N_1}$ is a Golomb system, the equality $\hpi_n(B)=B_n$ holds for every $n$ in $\N_1$, by \eqref{e:glmsyst4}. Besides, in this case $B$ has the form
\begin{equation} \label{e:glmbprd} B=\prod_{p\in\calp}B_{p^\infty}\,, \end{equation}
where $B_{p^\infty}=\hpi_{p^\infty}(B)$, as can be easily deduced comparing \eqref{e:glmsyst3} with \eqref{e:120}. Note also that \eqref{e:glmsyst2} implies
$$B_{p^\infty}=(\hpi_{p^{\gamma(p)}}^{p^\infty})^{-1}(B_{p^{\gamma(p)}})$$
(where $\hpi_{p^{\gamma(p)}}^{p^\infty}\colon\Z_p\rightarrow\Z/p^{\gamma(p)}\Z$ is the natural projection). Thus the hypothesis $\gamma(p)\neq\infty$ yields that $B_{p^\infty}$ is both open and 
closed in the profinite topology (because so is $B_{p^{\gamma(p)}}$ with respect to $\calt_{F,p^{\gamma(p)}}$).

Conversely, any choice of an open closed subset of $\Z_p$ for every $p$ gives rise to a Golomb system, by taking the product over all primes and then projecting to finite quotients. We leave to 
readers the simple task of checking details.

\begin{eg} In the cases discussed in Example \ref{eg:brwntyp}, we have $B=\prod_p\Z_p^*$ for $\calt_S$ and $\calt_R$ and $B=\prod_pp\Z_p$ for $\calt_G$ and $\calt_K$\,. As readers can 
check, these sets are respectively the units and the topologically nilpotent elements of $\za$. We also remark that in these cases $B\cap\iota(\Z)$ is either $\{\pm1\}$ or $\{0\}$, corresponding to 
the fact that one has to remove such points from $\Z$ in order to have a Hausdorff space. \end{eg}

In general, a p\'eij\'i topology $\calt$ can be extended to $\za$ as a restriction of the product topology $\prod\calt_n$ discussed in \S\ref{sss:cmptpj}: we shall denote the resulting space as 
$(\za,\calt)$. 

Recall that a point $x$ in a topological space $X$ is {\em indiscrete} if the only open subset containing $x$ is $X$ itself.

\begin{prop} \label{p:b2za} Assume $\calt$ satisfies condition {\bf B2} of Definition \ref{d:brwntyp} and define the subset $B$ of $\za$ as in Lemma \ref{l:clsdB2}. Then every $x\in B$ is an 
indiscrete point of $(\za,\calt)$. \end{prop}

\begin{prf} By \eqref{e:zasbst}, we can think of $x$ as $(\hpi_n(x))\in\prod\Z/n\Z$. Hence a neighbourhood of $x$ in $(\prod\Z/n\Z,\prod\calt_n)$ must be a product of neighbourhoods of the 
$\hpi_n(x)$'s. Now apply Corollary \ref{c:bbn} and recall that, by construction, the only neighbourhood of a point in $B_n$ is $\Z/n\Z$. \end{prf}

A topological space with an indiscrete point is trivially connected; Proposition \ref{p:b2za} shows that this applies to $(\za,\calt)$. This simple observation was the inspiration behind Theorem 
\ref{t:brwntyp} and the definition of Golomb systems.

\begin{rmk} It is instructive to compare the hypotheses of Theorem \ref{t:brwntyp} and Proposition \ref{p:b2za}: the former requires conditions {\bf B1} and {\bf B2} to establish that $(\Z,\calt)$ is 
connected, while the latter only uses {\bf B2} to deal with $(\za,\calt)$. For a better understanding of the role of {\bf B1}, we offer Example \ref{eg:b2} below.

We also note that for a general $\calt$ any $x\in\prod B_n$ is an indiscrete point of $(\prod\Z/n\Z,\prod\calt_n)$; condition {\bf B2} is used to ensure that some of these points lie in $\za$.
\end{rmk}

\begin{eg} \label{eg:b2} Fix $x\in\za$ and define each $\calt_n$ by taking all singletons different from $\{\hpi_n(x)\}$ as open (if $x=0$ then this is the topology $\calt_{KP}$ of \S\ref{sss:tkp}). 
One immediately sees that the resulting topology $\calt$ satisfies {\bf B2} (with $B=\{x\}$), but not {\bf B1}. The space $(\za,\calt)$ is connected, but it becomes totally separated by removing 
$x$. In order to see it, just observe that $\hpi_n(x)\neq\hpi_n(y)$ implies that $y+n\za=\hpi_n^{-1}(\hpi_n(y))$ is open and hence the restriction of $\calt$ to $\za-\{x\}$ is the same as the 
profinite topology. In particular, if $x\notin\Z$ then $(\Z,\calt)=(\Z,\calt_F)$.

A similar analysis applies to the topology $\calt_{zd}$ of \S\ref{sss:tzd}: in this case one finds $B=\za^*$ and $\za-B$ has the profinite topology. \end{eg}

\begin{qsts} \label{q:BG} The discussion above establishes a bijection between Golomb systems and closed subsets of $\za$  of the form \eqref{e:glmbprd}. In light of the many recent works discussing 
the homeomorphism problem for Golomb topologies (in particular \cite{sp}, \cite{bspt} and \cite{bst}), this suggests a number of questions. \begin{itemize}
\item[{\bf Q1.}] Let $\calb=\{B_n\}$ and $\calc=\{C_n\}$ be two Golomb systems, attached respectively to the closed subsets $B$ and $C$ in $\za$. In the notation of Remark \ref{r:ntz}, is it true 
that the spaces $(\Z,\calt_\calb)$ and $(\Z,\calt_\calc)$ are homeomorphic if and only if there is a homeomorphism $\phi$ of $\za$ in itself such that $\phi(\Z)=\Z$ and $\phi(B)=C$\,? 
\item[{\bf Q2.}] Assuming that {\bf Q1} admits a positive answer, is it true that two Kirch coarsenings $(\Z,\calt_{\calb,\kappa})$ and $(\Z,\calt_{\calc,\eta})$ are homeomorphic if and only if
$\phi$ moreover changes $\kappa$ into $\eta$?
\item[{\bf Q3.}] Replacing the unique factorization of integers with unique factorization of ideals, the construction of Golomb systems can be easily extended to any Dedekind domain $D$. Assuming 
that $D$ is residually finite (so that $\varprojlim D/\gotn$ is compact) and countable, under what conditions (if any) can $(\Z,\calt_\calb)$ be homeomorphic to $D$ with the topology induced by a 
Golomb system? \end{itemize}
The main results of \cite{bspt} and \cite{bst} give some reason to think that {\bf Q1} and {\bf Q2} could have a positive answer. \end{qsts}

\begin{rmk} \label{r:tree} There is a huge gap between the general notion of a p\'eij\'i topology of Brown type $\calt$ and our definition of $\calt_{\calb,\kappa}$ from a  Golomb system. By Lemma 
\ref{l:clsdB2} we know that $\calt$ always determines a closed subset $B$ in $\za$, but if $B$ has not the form \eqref{e:glmbprd} it was not immediately clear to us how to attach a topology to it. On 
the other hand, the discussion in subsection \ref{ss:BG} has a strong combinatorial flavour, which suggests the approach we are going to sketch. 

Let $\T$ be the rooted tree with $n!$ vertices at distance $n$ from the origin $v_0$ (so that $n+1$ edges depart from each vertex at distance $n$, for every $n\in\N_1$). The {\em ends} of $\T$ are 
infinite paths starting in $v_0$ and with no backtracking. Fixing compatible bijections between vertices at distance $n$ from $v_0$ and points in $\Z/n!\Z$, one obtains a combinatorial realization of 
$\za$ as the set of ends of $\T$. In particular, cosets of $n\za$, with $n\in\N_1$, correspond to the ends passing through a vertex $v$ at distance $n$ from the origin and this determines all subsets 
of $\za$ which are both open and closed.

We expect that a translation of the definition of $\calt_\calb$ (and the surrounding ideas) in terms of $\T$ would bring more light on the general notion of p\'eij\'i topologies of Brown type (and 
could help in answering Questions \ref{q:BG} and extending their scope). \end{rmk}

\subsection{An arithmetic application: the closure of prime numbers} \label{ss:clsP} 

\subsubsection{Units of $\za$ and prime numbers} By Theorem \ref{t:440} we have 
\begin{equation} \label{e:untza} \za^*=\prod_{p\in\calp}\Z_p^*\,. \end{equation}
As a consequence, we obtain the following proposition, which, in light of
$$\za^*\cap\Z=\Z^*=\{\pm1\}\;,$$
can be viewed as an extension of \cite{fur} (see also \cite[Theorem 5.1]{brou} and the paragraph ``Following Furstenberg'' of \cite[pages 202-203]{clrk} for similar results).

\begin{prop} \label{p:32} The interior of $\za^*$ is empty if and only if the set of primes $\calp$ is infinite. \end{prop}

\begin{proof} We use \eqref{e:untza}. By contradiction, if primes were finite, 
i.e., $\calp=\{2,3,...,p_m\}$, then $\prod_{p=2}^{p_m}\Z_p^*$ should be open, since $\Z_p^*$ is open, and obviously a finite product of open sets remains open.
	
Conversely, we recall that the infinite product of open sets is open if and only if a finite number of open sets is a proper subset. In our case $\Z_p^*$ is a proper subset for every prime $p$, thus 
we conclude. \end{proof}

\begin{rmk} \label{1000} The maps $\hpi_n$ are all continuous with respect to the topology $\calt_{F,n}$\,, which is Hausdorff. Hence the equality $\za^*=\bigcap_n\hpi_n^{-1}([1]_n)$ shows that 
$\za^*$ is closed in $\za$ (and therefore compact). \end{rmk}

For $X\subseteq\Z$ let $\supp(X)$ denote the set of primes dividing some element of $X$. Proposition \ref{p:32} has the following easy strengthening (which is just a reformulation of Euclid's 
classical argument).
 
\begin{prop} \label{p:ecld} If $X\cap\Z^*=\emptyset$ and $\xa\cap\za^*\neq\emptyset$, then $\supp(X)$ is infinite. \end{prop}

\begin{proof} By contradiction. If $\supp(X)$ is finite, then the set
$$C=\bigcup_{p\in\supp(X)}p\za$$ 
is closed in $\za$ (because each ideal $p\za$ is both open and closed) and it contains $X$ (since $X\cap\Z^*=\emptyset$), implying $\xa\subseteq C$. Now use $C\cap\za^*=\emptyset$. \end{proof}

The interesting point about Proposition \ref{p:ecld} is that, while $\Z^*$ contains only two points, $\za^*$ is quite big, providing many more opportunities for an intersection with $\xa$. This is 
nicely illustrated by the following theorem, which is one of the key inspirational ideas of our work. It is probably well-known to experts, but we are not aware of it appearing in the literature 
before (apart from a quick comment in \cite[Remark 3.3]{DL}). A generalization to rings of $S$-integers of global fields can be found in \cite[Theorem 3.2]{DL}.

\begin{thm} \label{t:23} Dirichlet's theorem on primes in arithmetic progressions holds if and only if
\begin{equation} \label{e:23} \hat\calp=\calp\cup\za^* \end{equation}
is true. \end{thm}

\begin{proof} First, if $x\in\za$ is not a unit then 
$$\za^*=\za-\bigcup_{p\in\calp}p\za$$
implies $x\in q\za$ for some prime $q$. Since $q\za$ is open and $q\za\cap\calp=\{q\}$, it follows that $x\in\hat\calp$ is true if and only if $x=q\in\calp$.

Now note that sets of the form $x+b\za$, $b\in\N_1$, form a base of neighbourhoods of $x\in\za$ and one has
$$x+b\za=a+b\za\,,$$
where $a$ is any integer satisfying $[a]_b=\hpi_b(x)$. If $x\in\za^*$ then $\hpi_b(x)$ is a unit in $\Z/b\Z$ and so $a,b$ are coprime; conversely, \eqref{e:ltrelr} yields that $(a,b)=1$ implies 
$(a+b\za)\cap\za^*\neq\emptyset$. Thus
\begin{equation} \label{e:drchza} \za^*\subset\hat\calp\iff(a+b\za)\cap\calp\neq\emptyset\text{ for all }a,b\text{ coprime.} \end{equation}
Since $(a+b\za)\cap\calp=(a+b\N)\cap\calp$, the right-hand side of \eqref{e:drchza} can be reformulated as the statement that, if $a$ and $b$ are coprime, there is a prime in the arithmetic 
progression $a+b\N$. But then there must be infinitely many, because $a+b\N$ contains $a+bc\N$ for all $c\in\N$ coprime with $a$. Therefore the left-hand side of \eqref{e:drchza} is equivalent to 
Dirichlet's theorem. \end{proof}

From \eqref{e:23} one immediately sees that the closure of $\calp$ in $\Z$ is $\calp\cup\Z^*=\calp\cup\{\pm1\}$. Thus we recover \cite[Theorem 4.2]{brou}, as was mentioned in \S\ref{sss:appltpF}.

\begin{rmk} \label{r:elrfrst} As well-known, Euler proved that the series $\sum_{p\in\calp}\frac{1}{p}$ diverges, thereby giving a new proof that there are infinitely many primes. His method can be 
summarized as evaluating at $1$ the function $\log\zeta(s)$, where $\zeta$ is the Riemann zeta function. This celebrated piece of work can be seen as the birth of analytic number theory, as 
Dirichlet's theorem was proved by an extension of Euler's technique.

Theorem \ref{t:23} provides a connection between Euler's and Furstenberg's proofs, by means of the Haar measure $\mu_{\za}$ on the additive group $\za$. Since $\calp$ is countable, \eqref{e:23} 
implies
$$\mu_{\za}(\hat\calp)=\mu_{\za}(\za^*)\stackrel{\star}{=}\prod_{p\in\calp}\left(1-\frac{1}{p}\right)=\frac{1}{\zeta(1)}=0$$
where $\star$ is an easy consequence of \eqref{e:untza} (see \cite[Example 4.3]{DL} for details). Thus, Furstenberg's argument is based on the fact that $\hat\calp$ has empty interior, while Euler's 
consists in computing its measure.

The equivalence \eqref{e:drchza} suggests the possibility of a proof of Dirichlet's theorem by independently establishing the inclusion $\za^*\subset\hat\calp$. The measure-theoretic considerations 
sketched above might give some indications on how to proceed. \end{rmk}

\subsubsection{Some observations about \cite[Section III]{gol2}} \label{sss:gol2S3} 

In the last section of \cite{gol2}, Golomb attaches to every arithmetic progression $A=a+b\N$ what he calls a ``measure'' $\pi$, by putting $\pi(A):=\frac{1}{\varphi(b)}$ (with $\varphi$ the 
Euler's totient, i.e., $\varphi(b):=|(\Z/b\Z)^*|$\,) if $a$ and $b$ are coprime and $\pi(A)=0$ otherwise. He also notes that the quantitative version of Dirichlet's theorem yields
\begin{equation} \label{e:glmdns} \pi(A)=\lim_{t\rightarrow\infty}\frac{\pi(t,A)}{t/\log t} \end{equation}
(where $\pi(t,X)$ counts the number of primes in $X$ up to $t$) and uses \eqref{e:glmdns} as definition of $\pi(X)$ for any $X\subseteq\N$. In particular, he calls $X$ ``sparse'' if $\pi(X)=0$. 
Finally, he compares this notion of sparsity with the property of having empty interior with respect to $\calt_G$ (finding out that the two are independent of each other, \cite[Theorems 12 and 
13]{gol2}) and gives an application to the set of primes dividing the Fermat numbers $2^{2^n}+1$.

We claim that Theorem \ref{t:23} sheds more light on Golomb's results. To start with, recall that, as a compact topological group, $\za^*$ is endowed with a Haar measure $\mu_{\za^*}$, normalized so 
to have total mass $1$. The closure in $\za$ of $A=a+b\N$ is $\hat A=a+b\za$ and one finds $\mu_{\za^*}(\hat A\cap\za^*)=\pi(A)$. Note also that, by the prime number theorem, the  right-hand side of 
\eqref{e:glmdns} is the asymptotic density of $A$ relative to $\calp$.

The comparison between density and Haar measure is a leading theme of \cite{DL}, in a setting where both take the value $0$ for $\calp$ (which is regarded as a subset of the much bigger 
set $\Z$). Theorem \ref{t:23} suggests that the same approach can be applied to $X\subseteq\calp$, by taking \eqref{e:glmdns} for the density and $\mu_{\za^*}(\xa\cap\za^*)$ for the measure. (When 
the 
limit in \eqref{e:glmdns} does not exist, replace it by $\limsup$ and $\liminf$, denoting the resulting values as $\pi^+(X)$ and $\pi^-(X)$ respectively.) By the same straightforward reasoning of 
\cite[Lemma 4.10]{DL} one can show 
$$\pi^+(X)\le\mu_{\za^*}(\xa\cap\za^*)$$
and use it to explain \cite[Theorems 12 and 13]{gol2}. We plan to provide full details in a future work.

Does the Golomb topology play a role in all of this? In our opinion, not really. On the one hand, the discussion in \cite[Section III]{gol2} seems to assign to $\calt_G$ only a minor role (in the 
comparison with  sparsity) and there is no suggestion of a stronger connection with the ``measure''. On the other one, invertible cosets are open in both $\calt_F$ and $\calt_G$ and thus one 
obtains the equality
$$\xa\cap\za^*=\xa^{\calt_G}\cap\za^*\,$$
which shows that considering $\xa^{\calt_G}$ would not change anything in the comparison with $\pi(X)$.
Actually, our analysis appears to support the statement that the best topology for arithmetic purposes is the Furstenberg one.

\section{The supernatural numbers} \label{s:sprnt}

\subsection{The supernatural numbers as a quotient of $\za$} We briefly recall Steinitz's concept of supernaturals, introduced in 1910 (\cite[page 250]{stn}, with the name of ``$G$-Zahlen'' 
\footnote{A shortening for ``Grad-Zahlen'',  translatable in English as ``degree numbers''.}), and its relation with our profinite setting. 

\begin{dfn} \label{d:sprnt} We define the {\em supernatural numbers} to be the set of formal products
$$\s:=\left\{  \prod_p p^{e_p} : e_p\in \tilde\N,\; p\in\calp\right\} , $$ 
where $\tilde\N:=\N\cup\{\infty\}$. The set $\s$ forms an abelian multiplicative monoid, with multiplication defined by exponent addition. \end{dfn}

The fundamental theorem of arithmetic suggests an obvious embedding of $\N_1$ into $\s$, $n\mapsto\prod_pp^{v_p(n)}$. This is extended to the whole $\N$ by $0\mapsto\prod_pp^\infty$. The notions 
of valuation at a prime $p$ and support have obvious extension from $\N$ to $\s$: if $s\in \s$, then $v_p(s)$ is the exponent of $p$ appearing in $s$ and the {\em support} of $s$ is the set 
$$\supp(s):=\{p\in\calp : v_p(s)\neq 0\}.$$ 
Moreover, for $s,t\in\s$, we say that $t$ {\em divides} $s$, in symbols $t|s$, if and only if $v_p(t)\le v_p(s)$ for every $p\in\calp$.

The set $\tilde\N$ is a monoid with addition (and the convention $x+\infty=\infty$ for every $x\in\tilde\N$).

\begin{lem} \label{l:sprdt} As monoids, we have 
\begin{equation} \label{e:sprdt} \s\simeq\prod_{p\in\calp} \tilde\N\,. \end{equation}
\end{lem}

\begin{proof} Writing $s\in\s $ as $s=\prod_p p^{v_p(s)}$, we map $s\mapsto(v_p(s))_{p\in\calp}$. \end{proof}

We now give the relation of $\s$ with $\za$.

\begin{prop} \label{p:99} There is a canonical bijection $\s\simeq \za/\za^*$. \end{prop}

\begin{proof} Using the valuations $\hvp$ introduced in \S\ref{sss:arthza}, define the map $\rho\colon\za\longrightarrow\s$ by
\begin{equation} \label{e:zasprnt} x\mapsto \prod_{p\in\calp}p^{\hvp(x)}\,. \end{equation}
By Theorem \ref{t:440}, for $x\in\za$ we can write $x=(x_p)$, with $x_p\in\Z_p$ and $\hvp(x)=v_p(x_p)$. Thus we see
$$\rho(x)=\rho(y)\Longleftrightarrow\forall\,p\in\calp\;\exists\,u_p\in\Z_p^*\text{ such that }x_p=u_pz_p\Longleftrightarrow\exists\,u\in\za^*\text{ such that }x=uy\,.$$
Now it is enough to observe that $\rho$ is onto (as immediate from Definition \ref{d:sprnt}).  \end{proof}

Note that the map \eqref{e:zasprnt} is compatible with the embeddings of $\Z$ and $\N$ into $\za$ and $\s$, respectively. 

\begin{rmk} \label{r:sprnid} It follows immediately from Proposition \ref{p:99} that $\s$ can be identified with the set of principal ideals of $\za$. It is not hard to show that an ideal of 
$\za$ is principal if and only if it is closed (a proof, in the more general context of Dedekind domains, can be found in \cite[Lemma 2.2]{DL}). Moreover, continuity of the ring operations implies 
that closed subgroups of $\za$ must be ideals. Hence the supernatural numbers classify the closed subgroups of $\za$ (which is essentially the reason why they were originally introduced: namely, 
closed subgroups of $\za$ correspond to algebraic extensions of finite fields, which are the objects that Steinitz wanted to study). \end{rmk}

\subsubsection{The topology of $\s$} By Proposition \ref{p:99} we can endow $\s$ with the quotient topology, so that it is compact and $\N$ embeds into it with dense image. 

We also provide a different perspective on this topology, following \cite{pol}. Give to $\N$ the discrete topology and make $\tilde\N$ its Alexandroff compactification (that is, a set is open in 
$\tilde\N$ if and only if it is contained in $\N$ or its complement is finite). Then one can use the isomorphism \eqref{e:sprdt} and give to $\s$ the product topology. 

\begin{lem} The quotient topology defined on $\s$ via \eqref{e:zasprnt} coincides with the product one induced by \eqref{e:sprdt}. \end{lem}

\begin{prf} This follows by proving that the map $\za\rightarrow\prod_p\tilde\N$, $x\mapsto(\hvp(x))_p$, is continuous (and hence closed, because $\za$ is compact and $\prod_p\tilde\N$ Hausdorff). By 
the universal property of the product, it is enough to show that each $\hvp$ is continuous. The valuation maps $v_p\colon\Z_p\rightarrow\tilde\N$ are all continuous, because the sets 
$v_p^{-1}(n)=p^n\Z_p^*$ are both open and closed if $n\in\N$. One concludes remembering that $\hpi_{p^\infty}$ is continuous and $\hvp=v_p\circ\hpi_{p^\infty}$\,. \end{prf}

As $\s$ is a compact topological space, any sequence in it contains some convergent subsequence (a sequence $(s_i)_{i\in\N}$ in $\s$ converges to $s\in\s$ if and only if we have 
$\lim_iv_p(s_i)=v_p(s)$ for every prime $p$). In particular this applies when the $s_i$'s are natural numbers: for example, it is interesting to observe that  any ordering of $\calp$ produces a 
sequence converging to $1$ in $\s$ (in $\za$ such a sequence will have many different accumulation points, as follows from Theorem \ref{t:23}). In \cite{pol} the sequential compactness of $\s$ is 
applied to prove various results about perfect, amicable and sociable numbers.

\subsubsection{Topologies on $\N$} \label{sss:tplgN} The obvious analogue of Proposition \ref{p:99} yields $\N=\Z/\Z^*$, allowing for the interpretation of $\N$ as the set of ideals of 
$\Z$, similar to $\s$ for $\za$. The relations can be summarized with the commutative diagram
\begin{equation} \label{e:cdsn} \begin{CD} \Z @>{\iota}>> \za \\ @V{\rho|_\Z}VV @VV{\rho}V \\ \N @>>> \s  \end{CD} \end{equation}
where the horizontal arrows are dense embeddings and the vertical ones projections.

We provide a little more detail on the topology $\calt_F^*$ that $\N$ obtains from $\rho$. For any $x\in\s$, write $x\s$ for its set of multiples.

\begin{lem} \label{l:bssprnt} The sets $\{n\s\}_{n\in\N_1}$ and their complements form a base for the topology of $\s$. \end{lem} 

\begin{prf} The action of $\za^*$ on $\za$ is continuous, because  the latter is a  topological ring. By definition of quotient topology, it follows that the map $\rho$ of \eqref{e:zasprnt} is open 
and it sends a base into a base. Thus we can obtain a base for the topology of $\s$ by means of the one described in \eqref{e:bsza} for $\za$. Now it is enough to observe that the equality 
$$\za^*\cdot\hpi_n^{-1}\big([a]_n\big)=\hpi_n^{-1}\big(\za^*\cdot[a]_n\big)=\hpi_n^{-1}\big((\Z/n\Z)^*\cdot[a]_n\big)$$
(where $\cdot$ denotes the multiplicative action) yields $\rho^{-1}\big(\rho(n\za)\big)=n\za$ and
\begin{equation} \label{e:orbtzast} \rho^{-1}\big(\rho(a+n\za)\big)=d\za-\bigcup_{p|n}p^{v_p(d)+1}\za \end{equation}
with $d$ the greatest common divisor of $a$ and $n$. (In other words, \eqref{e:orbtzast} amounts to the statement that the orbit of $[a]_n$ in $\Z/n\Z$ is given by the difference between the ideal 
generated by $[a]_n$ and every ideal contained in it.)  \end{prf}

It follows that $\calt_F^*$ can be thought of as a ``disconnected version'' of $\calt_R|_\N$ (in the sense that both $n\N$ and its complement are open in $\calt_F^*$). Note  also that 
this topology is coarser than $\calt_F|_\N$\,. A similar phenomenon with regard to $\calt_G$ was discussed in \cite[Proposition 11 and Example 11.1]{kp}. In general, given a topology $\calt$ on 
$\Z$, one can use both the inclusion $\N\hookrightarrow\Z$ or the projection $\Z\twoheadrightarrow\N$ to define topologies, respectively $\calt|_\N$ and $\calt^*$, on $\N$. The composition $\id_\N$ 
of the two maps is always continuous, showing that the inclusion induces a finer topology; usually one has $\calt|_\N\supsetneq\calt^*$. 

It is easy to see that if $\calt$ is p\'eij\'i and its maximal family $\{\calt_n\}$ is invariant under $x\mapsto-x$ then one has $\calt^*=\calt'|_\N$ for a p\'eij\'i topology $\calt'$. 
This obviously applies to $\calt_F$ and (as observed in \cite[page 948]{orum}) to $\calt_G$\,.

\begin{rmk} From the viewpoint of algebraic number theory, it is somewhat more ``natural'' to think of $\N$ as a set of ideals (along the lines of Remark \ref{r:sprnid}) rather than as a subset of 
$\Z$. The space $\s$ of closed ideals of $\za$ fits nicely in this perspective as the compactification of $\N$. \end{rmk}

\subsection{Some order-preserving functions} In this subsection we consider some arithmetic functions which can be extended to $\s$ and hence to $\za$.

\subsubsection{Order topologies and the order on $\s$} \label{sss:ordtplg} An order relation can be used to induce a topology. We will consider two cases, the extended real line 
$\tilde\R:=\R\cup\{\pm\infty\}$ and $\s$. For the order topology on $\tilde\R$, we take as open sets $\tilde\R$, $\emptyset$ and the half-lines $(r,\infty]$ with $r\in\R$. (Note that we can also 
consider $\tilde\R$ with the usual Euclidean topology - in this case, the Alexandroff compactification $\tilde\N$ is a subspace.)

As for $\s$, first note that it obtains an order from divisibility, $s\le t\Longleftrightarrow s|t$. Let $\rho$ be the map \eqref{e:zasprnt}. The {\em order topology} on $\s$ is the quotient 
along $\rho$ of the order topology on $\za$, which in turn is defined taking the set of principal ideals $\{x\za:x\in\za\}$ as base. Readers can check that on $\s$ this is the same as the 
topology defined by the closure operator
$$X\mapsto\overline X:=\bigcup_{t\in X}\{s\le t\}\,.$$
The restriction of this topology to $\N$ is exactly the one defined by Rizza in \cite[\S2]{rzz}.

\begin{lem} \label{l:ordcnt} Let $f\colon\s\rightarrow\tilde\R$ be a function such that $s|t$ implies $f(s)\le f(t)$. Then $f$ is continuous with respect to the order topologies on $\s$ and 
$\tilde\R$. \end{lem}

\begin{prf} Obvious from $(f\circ\rho)^{-1}\big((r,\infty]\big)=\bigcup x\za$, where the union is taken over those $x\in\za$ such that $f(\rho(x))>r$. \end{prf}

\begin{rmk} Lemma \ref{l:ordcnt} is a special case of the well-known result that a function between posets is continuous if and only if it is order-preserving (see e.g. \cite[Theorem 
2.1]{hauk}). \end{rmk}

The order topology on $\za$ is not comparable with the profinite one. However, remembering Remark \ref{r:rzz}, we extend $\calt_R$ to $\za$ by taking the ideals $\{n\za:n\in\N_1\}$ as a base. 
This topology (which, by abuse of notation, will also be denoted as $\calt_R$) is coarser than both the profinite and the order topology. On $\s$ it induces a topology with base 
$\{n\s:n\in\N_1\}$.

\begin{rmk} \label{r:ordcnt} Let $f$ be as in Lemma \ref{l:ordcnt}, with the additional property that for every $r\in\R$ there is a set $S(r)\subseteq\N_1$ such that $(f\circ\rho)(x)>r$ implies 
$x\in\bigcup_{n\in S(r)}n\za$. Then the same proof as above shows that $f\circ\rho$ is continuous with respect to $\calt_R$ (and hence to the profinite topology). \end{rmk}

\subsubsection{Counting prime divisors} We extend to $\za$ two well-known functions $\N_1\rightarrow\N$, by \begin{itemize}
\item $\omega\colon\za\longrightarrow\N\cup\{\infty\}\,,\;\;\;x\mapsto |\supp(x)|=|\{p\in\calp : v_p(x)\neq 0 \}|$\,;
\item $\Omega\colon\za\longrightarrow\N\cup\{\infty\}\,,\;\;\;x\mapsto \sum_p v_p(x)$. \end{itemize}
These functions are $\za^*$-invariant, so by Proposition \ref{p:99} we see they factor through $\s$. The following remarks show some immediate arithmetic interest: \begin{itemize}
\item an integer $n$ is prime if and only if $\Omega(n)=1$;
\item $\Omega^{-1}(0)=\za^*$ and $\Omega^{-1}(1)=\coprod_{p \in\calp}p\za^*$;
\item it follows from the definitions that $\Omega(x)=\infty$ if and only if either $\omega(x)=\infty$ or there is a prime $p$ such that $v_p(x)=\infty$, i.e., if and only if 
$\rho(x)\notin\rho(\N_1)$. \end{itemize}

\begin{prop} \label{p:98} Let $\alpha$ be either $\omega$ or $\Omega$. 
For $n\in\N$, consider $I_n:=\{0,...,n\}\subset \N$. Then \begin{itemize}
\item $\alpha$ is continuous with respect to the profinite topology on $\za$ and the order topology on $\tilde\R$;
\item for any $n\in\N$, the closure of $\alpha^{-1}(n)$ is $\alpha^{-1}(I_n)$;
\item $\alpha^{-1}(\infty)$ and its complement are both dense in $\za$. \end{itemize} \end{prop}

\begin{prf} One sees immediately that $\alpha$ enjoys the properties
$$x|y\Longrightarrow\alpha(x)\le\alpha(y)$$ 
and
\begin{equation} \label{e:omgN} \alpha(x)\in\N\Longleftrightarrow x\in\coprod_{n\in\N_1}n\za^* \end{equation}
We can think of $\alpha$ as a function from $\s$ to $\tilde\R$ and then the reasoning of Remark \ref{r:ordcnt}, with 
$$S(r)=\{k\in\N_1\mid\alpha(k)>r\}\,,$$
shows that $\alpha$ is continuous. Hence $\alpha^{-1}(I_n)=\alpha^{-1}([-\infty,n])$ is closed. 

In order to check that $\alpha^{-1}(n)$ is dense in $\alpha^{-1}(I_n)$, take $x\in\alpha^{-1}(n-1)$ and let $\calp'$ be the complement in $\calp$ of $\supp(x)$. Then $\alpha(px)=\alpha(x)+1=n$ for 
every $p\in\calp'$. Since $\calp-\calp'=\supp(x)$ is finite, Theorem \ref{t:23} implies that $\calp'$ contains a sequence $p_i$ converging to $1$ in $\za$ and we see that $x=\lim p_ix$ is an 
accumulation point of $\alpha^{-1}(n)$. This proves that $\alpha^{-1}(n-1)$ is in the closure of $\alpha^{-1}(n)$ and a recursive reasoning shows that the same is true for $\alpha^{-1}(n-k)$, 
$k=2,\dots,n$. 

In order to prove that both a subset and its complement are dense, it is enough to show that one of them is dense with empty interior. From \eqref{e:omgN} it follows that the complement of 
$\alpha^{-1}(\infty)$ is $\coprod_{n\in\N_1}n\za^*$, i.e., a countable union of closed sets, each of them with empty interior by Proposition \ref{p:32}. Since $\za$ is compact and metrizable, Baire's 
lemma applies, proving that the interior of $\coprod_nn\za^*$ is empty as well. Finally, $\coprod_nn\za^*$ is dense in $\za$ because it contains $\Z-\{0\}$.  \end{prf}

Proposition \ref{p:98} subsumes \cite[Theorem 4.3]{brou} and the first statement of \cite[Erratum, Proposition]{pol}. It also implies that neither $\omega$ nor $\Omega$ are continuous maps with 
respect to the 
Euclidean topology on $\tilde\R$. A more interesting consequence is a one-line proof of the fact that the set of integers supported at most in $n$ primes has density $0$, for any $n\in\N$ and any 
``reasonable'' definition of density. Namely, Proposition \ref{p:98} implies that the closure of $\omega^{-1}(n)$ has Haar measure $0$ in $\za$, because, by \eqref{e:omgN}, it can be written as a 
countable union of sets of the form $k\za^*$. The statement on density then follows from \cite[inequality (41)]{DL} (see also \cite[\S5.2.2]{DL} for the extension to rings of $S$-integers).

\subsubsection{The abundancy index} The abundancy index of a positive integer $n$ is defined by the abundance function 
$$h(n):=\frac{\sigma(n)}{n}=\frac{1}{n}\sum_{d|n}d$$ 
(where $\sigma$ denotes  the sum-of-divisors function). We refer to \cite{laa} for an elementary introduction. As well-known (and easy to prove), one has 
$$h(p^r)=\frac{p^{r+1}-1}{p^r(p-1)}\;\;\text{ for }p\in\calp,r\in\N_1$$
and $h$ is multiplicative, that is $h(ab)=h(a)h(b)$ if $a,b$ are coprime.

\begin{prop} \label{p:hPllLt} The abundancy index can be extended to a function $h\colon\s\rightarrow\tilde\R$, which is continuous with respect to the order topologies. Moreover, one has 
$h(\s)=[1,\infty]$. \end{prop}

\begin{prf} For $s=\prod_pp^{e_p}\in\s$, put
$$h(s):=\prod_{p\in\calp}h(p^{e_p})=\lim_{l\rightarrow\infty}\prod_{p<l}h(p^{e_p})\,,$$
with
$$h(p^\infty):=\frac{p}{p-1}=\lim_{r\rightarrow\infty}h(p^r)$$
(where the limits are taken with respect to the usual topology on $\R\cup\{\pm\infty\}$). Continuity follows from Lemma \ref{l:ordcnt}. The last statement is an easy consequence of the 
fact that the values $h(p)$ are arbitrarily close to $1$, combined with divergence of the product $\prod_{p\in\calp}h(p)$: see \cite[Theorem 5]{laa} for details. \end{prf}

Most of Proposition \ref{p:hPllLt} was already in the literature: the extension of $h$ to all of $\s$ is a key idea in \cite{pol} (note a mistake about the continuity of $h$ in the original paper, 
corrected in the erratum) and the fact that $h(\N_1)$ is dense in the half-line $[1,\infty]$ was first proved in \cite{laa}. Note also that our reasoning proves that $[1,\infty]$ is already covered 
by the squarefree elements of $\s$.

\end{document}